\documentclass[11pt]{scrartcl}

\usepackage{amsfonts,amsmath,amsthm,amssymb}
\usepackage{mathrsfs,mathtools}
\usepackage{enumerate}
\usepackage{enumitem}
\usepackage{hyperref}
\usepackage{esint}
\usepackage{graphicx}
\usepackage{bm}
\usepackage{commath}
\usepackage{esint}
\usepackage{color}

\newcommand{\abssec}[1]{\noindent\normalsize {\bfseries #1\quad }\ignorespaces}
\renewenvironment{abstract}{\abssec{Abstract}}{\par\vspace{.1in}}
\newenvironment{keywords}{\abssec{Key Words}}{\par\vspace{.1in}}
\newenvironment{AMS}{\abssec{AMS subject
		classification}}{\par\vspace{.1in}}

\numberwithin{equation}{section}

\theoremstyle{plain}
\newtheorem{theorem}{Theorem}[section]

\newtheorem{lemma}[theorem]{Lemma}

\theoremstyle{definition}
\newtheorem{remark}[theorem]{Remark}
\newtheorem{example}[theorem]{Example}

\newtheorem{definition}[theorem]{Definition}

\newtheorem{assumption}[theorem]{Assumption}

\DeclareMathOperator{\proj}{Proj}

\newcommand{\reg}{\mathrm{r}}
\newcommand{\sing}{\mathrm{s}}

\title{Error estimates for Dirichlet control problems in polygonal domains
\thanks{The project was supported by DFG through the International Research Training Group IGDK 1754 \emph{Optimization and Numerical Analysis for Partial Differential Equations with Nonsmooth Structures}.
The second author was partially supported by the Spanish Ministerio Espa\~nol de Econom\'{\i}a y Competitividad under research project MTM2014-57531-P}}

\author{Thomas Apel \thanks{Institut f\"ur Mathematik und Bauinformatik, Universit\"at der Bundeswehr M\"unchen,85577 Neubiberg, Germany. \texttt{Thomas.Apel@unibw.de}}
\and
Mariano Mateos \thanks{Departamento de Matem\'{a}ticas, Universidad de Oviedo, 33203, Gij\'on, Spain. \texttt{mmateos@uniovi.es}}
\and
Johannes Pfefferer \thanks{Lehrstuhl f\"ur Optimalsteuerung, Technische Universit\"at M\"unchen, 85748 Garching bei M\"unchen, Germany. \texttt{pfefferer@ma.tum.de}}
\and
Arnd R\"osch \thanks{Fakult\"at f\"ur Mathematik, Universt\"at Duisburg-Essen, 45127 Essen, Germany. \texttt{Arnd.Roesch@uni-due.de}}
}

\begin{document}

\maketitle

\begin{abstract}
	The paper deals with finite element approximations of elliptic Dirichlet boundary control problems posed on two-dimensional polygonal domains. Error estimates are derived for the approximation of the control and the state variables. Special features of unconstrained and control constrained problems as well as general quasi-uniform meshes and superconvergence meshes are carefully elaborated. Compared to existing results, the convergence rates for the control variable are not only improved but also fully explain the observed orders of convergence in the literature. Moreover, for the first time, results in non-convex domains are provided.
\end{abstract}

\begin{AMS}
	65N30, 65N15, 49M05, 49M25
\end{AMS}

\begin{keywords}
	optimal control,  boundary control, Dirichlet control, nonconvex domain, finite elements, error estimates, superconvergence meshes
\end{keywords}

\section{Introduction}
\label{S1}

In this paper we will study the finite element approximation of the control problem
\[
\mbox{(P)}\left\{\begin{array}{l} \min
J(u)={\frac{1}{2}\|Su-y_\Omega\|_{L^2(\Omega)}^2
 + \frac{\nu}{2}\|u\|_{L^2(\Gamma)}^2}\\[1ex]
    \mbox{subject to} \ \
    (Su,u)\in H^{1/2}(\Omega)\times
    L^2(\Gamma),\\[1ex]
    u\in U_{ad}=\{u\in L^2(\Gamma):\ a\leq u(x)\leq b \ \ \mbox{ for a.a. } \ x\in \Gamma\},
    \end{array}\right.
\]
where $Su$ is the very weak solution $y$ of the
state equation
\begin{equation}
 -\Delta y   =  0  \mbox{ in }\Omega,\
 y = u  \mbox{ on } \Gamma,
\label{E1.1}
\end{equation}
the domain $\Omega\subset\mathbb{R}^2$ is bounded and polygonal,
$\Gamma$ is its boundary, $a<b$ and $\nu>0$ are real constants,
 and $y_\Omega$
is a function whose precise regularity will be stated when necessary. We assume
that $0\in[a,b]$ and comment on the opposite case in {Remark} \ref{sec:rem_ext}.
{Abusing notation, we will allow the case $a=-\infty$ and $b=+\infty$ to
denote the absence of one or both of the control constraints.}

First order optimality conditions read as (see \cite[Lemma 3.1]{AMPR-2015})
\begin{lemma}\label{T3.1}Suppose $y_\Omega\in L^2(\Omega)$. Then problem (P) has a unique solution $\bar u\in L^2(\Gamma)$ with related state $\bar y\in H^{1/2}(\Omega)$ and adjoint state $\bar\varphi\in H^1_0(\Omega)$. The following optimality system is satisfied:
\begin{subequations}
\begin{align}
(\nu\bar u-\partial_n\bar\varphi,u-\bar u)_{L^2(\Gamma)}& \geq 0\ \forall u\in U_{ad}
\label{E3.1}\\
-\Delta \bar y = 0\mbox{ in }\Omega,\; \bar y=\bar u&\mbox{ on }\Gamma, {\text{ in the very weak sense,}}
\label{E3.2}\\
-\Delta\bar \varphi = \bar y-y_\Omega \mbox{ in }\Omega,\; \bar \varphi=0&\mbox{ on }\Gamma,  {\text{ in the weak sense}}.
\label{E3.3}
\end{align}
\end{subequations}
\end{lemma}
The variational inequality \eqref{E3.1} is equivalent to
\begin{equation}\label{eq:proj}\bar u(x) =
\proj_{[a,b]}\left(\frac{1}{\nu}\partial_n\bar\varphi(x)\right)\mbox{ for a.e.
}x\in \Gamma,\end{equation}
where $\proj_{[a,b]}$ denotes the pointwise projection on the interval
$[a,b]$.

The aim of this paper is to investigate a finite element solution
  of the system \eqref{E3.1}--\eqref{E3.3}, in particular to derive
  discretization error estimates. The precise description of the
  regularity of the solution of the first order optimality system is
  an important ingredient of such estimates. They were proven in our
  previous paper \cite{AMPR-2015}; we recall these results in Section
  \ref{sec:2}. There were two interesting observations which we may illustrate
  in the following example.
\begin{example}
  Consider the L-shaped domain. The $270^\circ$ angle leads in
    general to a singularity of type $r^{2/3}$ in the solution of the
    adjoint equation; the regularity can be characterized by
    $\bar\varphi\in H^s(\Omega)$ with $s<\frac23$.  Hence, the control
    has a $r^{-1/3}$-singularity in the unconstrained case, $\bar u\in
    H^s(\Gamma)$ for all $s<\frac16$.
    In the constrained
    case, however, the control is in general constant in the vicinity
    of the singular corner since the normal derivative of the adjoint state has a pole there,
    we get $\bar u\in H^s(\Gamma)$ for all $s<\frac32$. This regularity is determined by the larges convex angle
    and by the kinks due to the constraints.

    Unfortunately, this is not the whole truth. In exceptional
    cases, e.\,g.\ when the data enjoy certain symmetry, the leading
    singularity of type $r^{2/3}$ may not appear in the adjoint state.
    Instead, the solution may have a $r^{4/3}$-singularity whose
    normal derivative has a $r^{1/3}$-singularity which is not
    flattened by the projection $\proj_{[a,b]}$. The control is less
    regular, $\bar u\in H^s(\Gamma)$ for
    all $s<\frac56$. See Example 3.6 in \cite{AMPR-2015}.
\end{example}
Hence, dealing with these exceptional cases is not fun but necessary. If in the unconstrained case a stress intensity factor vanishes
, i.e., the leading singularity does not occur, then the convergence result is still true, one may only see a better convergence in numerical tests. See Figure \ref{F:unc_exp}, right hand side and Remark \ref{R-unc}. However, in the constrained case, the situation is the opposite. The exceptional case leads to the worst-case estimate. To deal with the ``worst-case'' and the ``usual-case'' in an unified way, we introduce in \eqref{MME7} some numbers related to the singular exponents.

We distinguish two cases for the investigation of the
  discretization errors. After proving a general result in Section
  \ref{sec:3} we study the unconstrained case in Section \ref{sec:4}
  and the constrained case in Section \ref{sec:5}. We focus on
  quasi-uniform meshes and distinguish general meshes and
  certain superconvergence meshes.
  In order not to overload the
  present paper, we postpone the study of graded meshes to \cite{AMPR-2017}.
   The numerical
  tests
  in Section \ref{sec:num_exp} confirm the theoretical results.

The study of error estimates for Dirichlet control problems posed on
polygonal domains can be traced back to \cite{Casas-Raymond2006}, where a
control constrained problem governed by a semilinear elliptic equation posed in
a convex polygonal domain is studied. An order of convergence of $h^{s}$ is
proved for all
$s<\min(1,\pi/(2\omega_1))$, where $\omega_1$ is the largest interior
angle, in both the control and the state variable. Later, in
\cite{MayRannacherVexler2013}, it is proven that for unconstrained linear
problems posed on convex domains, the state variable exhibits a better
convergence property.
The corresponding proof is based on a duality
argument and estimates for the controls in weaker norms than $L^2(\Gamma)$.
However, to the best of our knowledge, the argumentation is restricted to
unconstrained problems. For the error of the controls in $L^2(\Gamma)$,
the order shown in
\cite{Casas-Raymond2006} is not improved.

Nevertheless, the regularity of the control and the
existing numerical experiments, see
\cite{MayRannacherVexler2013,Mateos-Neitzel2015}, suggested that for the control
variable the order should be greater:
$h^{s}$ for all $s<\min(1,\pi/\omega_1-1/2)$ if one uses standard quasi-uniform meshes, and for all $s<\min(3/2,\pi/\omega_1-1/2)$ if one uses certain quasi-uniform meshes which allow for superconvergence effects, see Definition \ref{def:superconvergence}. Our
main results, Theorems \ref{main:unc} and \ref{T5main}, fully explain
the observed orders of convergence in the literature for the control variable,
improve existing results for the state variable in constrained
linear-quadratic problems posed in convex domains, and provide the first available results in nonconvex domains.

\section{Notation and regularity results}\label{sec:2}

Let us denote by $M$ the number of sides of $\Gamma$ and $\{x_j\}_{j=1}^M$ its vertexes, ordered counterclockwise. For convenience denote also $x_0=x_M$ and $x_{M+1}=x_1$. We will denote by $\Gamma_j$ the side of $\Gamma$ connecting $x_{j}$ and $x_{j+1}$, and by $\omega_j\in (0,2\pi)$ the angle interior to $\Omega$ at $x_j$, i.e., the angle defined by $\Gamma_{j}$ and $\Gamma_{j-1}$, measured counterclockwise. Notice that  $\Gamma_{0}=\Gamma_M$. We will use $(r_j,\theta_j)$ as local polar coordinates at $x_j$, with $r_j=|x-x_j|$ and $\theta_j$ the angle defined by $\Gamma_j$ and the segment $[x_j,x]$.
In order to describe the regularity of the functions near the corners, we will introduce for every $j=1,\ldots,M$  a positive number $R_j$ and an  infinitely differentiable cut-off function $\xi_j:\mathbb{R}^2\to[0,1]$ such that the sets
\[N_j=\{x\in\mathbb{R}^2:0<r_j<2R_j,\,0<\theta_j<\omega_j\},\]
satisfy $N_j\subset\Omega$ for all $j$ and $N_i\cap N_j=\emptyset$ if $i\neq j$ and $\xi_j\equiv 1$ in the set $\{x\in\mathbb{R}^2:r_j<R_j\}$, $\xi_j\equiv 0$ in the set $\{x\in\mathbb{R}^2:r_j>2R_j\}$.

For every $j=1,\ldots,M$ we will call $\lambda_j$ the \emph{in general} leading singular exponent associated with the operator corresponding to the corner $x_j$. For the Laplace operator it is well known that $\lambda_j=\pi/\omega_j$.
Since in general the regularity of the solution of a boundary value problem
depends on the smallest singular exponent, it is customary to denote
\begin{equation}\label{eq:lambda1}
{\lambda} =\min\{\lambda_j:\
j=1,\ldots,M\}
\mbox{ and }p_D =\frac{2}{1-\min\{1,
{\lambda}\}}.\end{equation}

Our main estimates are for data $y_\Omega\in W^{1,p^*}(\Omega)$ for some $p^*>2$. To get these estimates it is key to use the sharp regularity results of the optimal control, state and adjoint state provided in \cite{AMPR-2015}. For both the control and the state it is enough to know the Hilbert Sobolev-Slobodetski\u\i{} space they belong to, but for the adjoint state we will need to know with some more detail the development in terms of powers of the singular exponents. To write this development, we must proceed in two steps in order to be able to define the \emph{effectively} leading singularity in each corner.

Our first result concerns the regularity of the adjoint state and is a consequence of \cite[Theorem 3.2 and Theorem 5.1]{AMPR-2015}.
For $m\in \mathbb{Z}$, ${t\in\mathbb{R}}$ and $1< p \leq +\infty$
we define
\begin{equation}\label{E2.3}
{\mathbb{J}^m_{t,p}=\left\{j\in\{1,\ldots,M\}\mbox{ such that }
0<m\lambda_j < 2+t-\frac{2}{p}\mbox{ and } m\lambda_j\notin\mathbb{Z}\right\}}.\end{equation}

\begin{lemma}\label{L3.1}Suppose $y_\Omega\in L^{\infty}(\Omega)$. Let
$\bar\varphi\in H^1_0(\Omega)$ be the optimal adjoint state, solution of
\eqref{E3.3}. Then, there exist a unique function $\bar\varphi_{\reg}\in
W^{2,p}(\Omega)$ and unique real numbers $( c_{j,m})_{j\in
{\mathbb{J}_{0,p}^m}}$,  for all $p<+\infty$ for constrained problems and
$p<p_D$ for unconstrained problems, such that
\begin{equation}\label{E3.5}\bar\varphi = \bar\varphi_{\reg} +\sum_{m=1}^3\sum_{j\in{\mathbb{J}_{0,p}^m}} c_{j,m}\xi_j r_j^{m\lambda_j}\sin(m\lambda_j\theta_j).\end{equation}
\end{lemma}
Note that $p_D=+\infty$ in convex domains such that we obtain for
constrained as well as for unconstrained problems the same regularity of the
optimal adjoint state. However, in non-convex domains, the control and hence
the state, as part of the right hand side of the adjoint equation, may be
unbounded in the unconstrained case, which leads to the restriction $p<p_D$ for
the regularity of $\bar{\varphi}_{\reg}$.
Moreover, it may
happen that the
effectively leading singularity corresponding to a corner $x_j$ is not the
first one. This means that the associated coefficient $c_{j,1}$ in the
asymptotic representation \eqref{E3.5} is equal to zero. However, this will be
of interest only for constrained problems in case of nonconvex corners $x_j$,
i.e., $\lambda_j<1$.
To be able to cover this, we define the numbers
\begin{equation}\label{MME7}
	\Lambda_j = \left\{\begin{array}{cl}
	\lambda_j&\mbox{ if }\lambda_j>1\mbox{ or }c_{j,1}\neq 0\\
	2\lambda_j&\mbox{ if }\lambda_j < 1\mbox{ and }c_{j,1}=0
	\end{array}\right.
\end{equation}
for each corner. In addition, we introduce
\begin{align}
\label{eq:Lambda} {\Lambda}&={\min\{\Lambda_j: \Lambda_j>1, \ j=1,\ldots,M\}}.
\end{align}
In convex domains,
$\lambda=\Lambda$ will determine the regularity
of both the optimal control and state. This holds for unconstrained as
well as for constrained problems. However, in nonconvex domains, different
cases may appear. If we have no control constraints then the regularity of the
optimal control and state will again be determined by $\lambda$. If the problem
is constrained then in the vicinity of any corner $x_j$, where the coefficient
of the corresponding first singularity $c_{j,1}$ is unequal to zero,
the optimal control is flattened there due to the projection formula and
consequently smooth. This is the usual case. If $c_{j,1}=0$ then the optimal
control in the neighborhood of such a corner is at least as regular as the
normal derivative of the corresponding second singular function.
In the control constrained case,
$\Lambda$
will determine the regularity of the
optimal control, at least in a worst case sense. The regularity of the optimal state may depend on $\lambda$
as well since singular terms may occur within its asymptotic representation
independent of the adjoint state.

For unconstrained problems the following regularity result holds, see
\cite[Corollary 5.3, Corollary 4.2, Theorem 3.4]{AMPR-2015}.
\begin{lemma}[unconstrained case]\label{L2.3}Suppose $-a=b=\infty$ and
$y_\Omega\in H^{t}(\Omega)
{\cap L^2(\Omega)}$ for all $t<
{\min\{1,\lambda-1\}}$. Then
\begin{align}
\bar u\in H^{s}(\Gamma),  \quad \bar y\in H^{s+\tfrac12}(\Omega)\quad\forall
s<\min\{\tfrac32,\lambda-\tfrac12\}
\end{align}
\end{lemma}

For constrained problems, we can improve this result, {see}
\cite[Corollary 4.2, Theorem~3.4]{AMPR-2015}.
\begin{lemma}[control constrained case]\label{L2.4}Suppose $-\infty<
a<b<\infty$, $y_\Omega\in H^{t}(\Omega)
{\cap L^2(\Omega)}$ for all
$t<{\min\{1,\lambda-1\}}$. Assume that the
optimal control has a finite number of kink points.
Then
\begin{align}
&\bar u\in H^{s}(\Gamma)&&\forall s<\min\{\tfrac32,\Lambda-\tfrac12\},\\
&\bar y\in H^{s+\tfrac12}(\Omega)&&
\forall s<\min\{\tfrac32,\Lambda-\tfrac12,
{\lambda}+\tfrac12\}.
\end{align}
\end{lemma}

We also have the following result from \cite[Proof of Theorem 3.4]{AMPR-2015}.
\begin{lemma}\label{L2.5}Suppose $-\infty< a<b<\infty$ and $y_\Omega\in L^2(\Omega)$. If $\lambda_j<1$ and $c_{j,1}\neq 0$, then one of the control constraints is active near the corner $x_j$, i.e.,
there exists $\rho_j>0$ such that for $x\in\Gamma$ with $|x-x_j|<\rho_j$ either
$\bar u\equiv a$ or $\bar u\equiv b$.
\end{lemma}

Finally, we can write the representation of the adjoint state for regular enough data.
For $m\in \mathbb{Z}$, ${t\in\mathbb{R}}$ and $1< p \leq +\infty$
we will also need
\begin{equation}\label{eq:L}
\mathbb{L}^m_{t,p}=\left\{j\in\{1,\ldots,M\}\mbox{ such that }
0<m\lambda_j < 2+t-\frac{2}{p}\mbox{ and }
m\lambda_j\in\mathbb{Z}\right\}.\end{equation}
The following result is a consequence of \cite[Corollary 4.4]{AMPR-2015}.
\begin{lemma}
Suppose
that $\Omega$ is convex or $-\infty<a<b<+\infty$, and that $y_\Omega\in W^{1,p^*}(\Omega)$ with $p^*>2$. Then, for $p>2$ such that
\[
\frac{3p-2}{\lambda_j p}\not\in \mathbb{Z}\ \mbox{ for all } j\in \{1,\ldots,M\}.
\]
and
\[p\leq p^*,\ p<p_D,\ p<\frac{2}{2-\min\{\lambda,2\}}\]
there exist a unique function $\bar\varphi_{\reg}\in W^{3,p}(\Omega)$ and unique real numbers $( c_{j,m})_{\mathbb{J}^m_{1,p}}$ and $(d_{j,m})_{\mathbb{L}^m_{1,p}}$, such that
\begin{align*}
\bar\varphi & = \bar\varphi_{\reg} + \sum_{m=1}^5\sum_{j\in \mathbb{J}^m_{1,p} }c_{j,m} \xi_j r_j^{m\lambda_j}\sin(m\lambda_j\theta_j) \\
&+ \sum_{m=1,3}\sum_{j\in \mathbb{L}^m_{1,p} } d_{j,m}\xi_j  r_j^{2} \left(\log(r_j)\sin(2\theta_j) + \theta_j\cos(2\theta_j) \right).
\end{align*}
\label{C4.1r}
\end{lemma}
Notice that the coefficients $c_{j,m}$ that appear in both expansions in
Lemmata \ref{L3.1} and \ref{C4.1r} coincide, due to the uniqueness of the
expansion. In the expansion of Lemma \ref{C4.1r} new terms appear that belong
to $W^{2,p}(\Omega)$ for all $p<+\infty$ but not to $W^{3,p}(\Omega)$ for $p>2$
satisfying the conditions in Lemma \ref{C4.1r}.

\section{A general discretization error estimate}\label{sec:3}
In this section we will present a general discretization error estimate in Theorem \ref{T3.2}.
The terms in this general estimate have to be estimated in particular cases.
This work will be done in later sections.

For the discretization, consider a family of regular triangulations $\{\mathcal{T}_h\}$ depending on a
mesh parameter $h$ in the sense of Ciarlet \cite{Ciarlet91}. Notice, that a triangulation $\mathcal{E}_h$ of the boundary is naturally induced by $\mathcal{T}_h$. We assume that the
space $Y_h$ is the space of
conforming piecewise linear finite elements. The space $U_h$ is the space of piecewise linear functions generated by the trace of elements
of $Y_h$ on the boundary $\Gamma$. We denote the subspace of $Y_h$
with vanishing boundary values by $Y_{0,h}$.

We also introduce the discrete solution operator $S_h:U\rightarrow Y_h$. For
$u\in U$ the function $S_hu\in Y_h$ is defined as the unique solution of
\begin{equation}\label{discharmonic}
  (\nabla S_hu,\nabla z_h)_{L^2(\Omega)}=0 \quad \forall z_h\in Y_{0,h} \mbox{ and }  (S_hu-u,v_h)_{L^2(\Gamma)}= 0\ \forall v_h\in U_h.
\end{equation}
We emphasize that on the boundary $S_hu$ coincides with the $L^2$-projection of
$u$ on $U_h$. Thus we get $S_h u_h= u_h$ on $\Gamma$ for $u_h\in U_h$.
Notice as well that \eqref{discharmonic} is not a conforming discretization
of the very weak formulation of the state equation. However, according to
\cite{Apel-Nicaise-Pfefferer2014,Berggren2004}, its applicability is
guaranteed.

In our discretized optimal control problem we aim to minimize the
objective function
\[
\mbox{($P_h$)}\left\{\begin{array}{l}\displaystyle \min
J_h(u_h)=\frac{1}{2}{\|S_hu_h-y_\Omega\|_{L^2(\Omega)}^2}
 + \frac{\nu}{2}{\|u_h\|_{L^2(\Gamma)}^2}\\[1ex]
    \mbox{subject to} \ \
        u_h\in U_{ad}^h:=\{u_h\in U_h:\ a\leq u_h(x)\leq b \ \ \mbox{ for all } \ x\in \Gamma\}.
    \end{array}\right.
\]
The first order optimality conditions of this problem were derived in \cite{Casas-Raymond2006} and are stated in the next lemma.

\begin{lemma}Problem $(P_h)$ has a unique solution $\bar u_h\in U_{ad}^h$, with related discrete state $\bar y_h=S_h \bar u_h$ and adjoint state $\bar \varphi_h$. The following discrete optimality system is satisfied
\begin{subequations}
\begin{align}
(\nu\bar u_h-\partial_n^h \bar \varphi_h,u_h-\bar u_h)_{L^2(\Gamma)}& \ge  0
 \mbox{ for all } u_h\in U_{ad}^h \label{VI}\\
(\nabla \bar y_h,\nabla z_h)_{L^2(\Omega)}&=0 \mbox{ for all } z_h\in Y_{0,h}\,
\mbox{ and } \bar
y_h|_\Gamma=\bar u_h, \label{discretestate}
\\
(\nabla \bar \varphi_h,\nabla z_h)_{L^2(\Omega)}&=(\bar y_h-y_\Omega,z_h)_{L^2(\Omega)}
\mbox{ for all } z_h\in Y_{0,h},
\label{discreadjoint}
\end{align}
\end{subequations}
where the discrete normal derivative $\partial_n^h \bar\varphi_h\in U_h$ is defined
as the unique solution  of
\begin{equation}
(\partial_n^h \bar\varphi_h,z_h)_{L^2(\Gamma)}=-(\bar y_h-y_\Omega,z_h)_{L^2(\Omega)}+
(\nabla \bar \varphi_h,\nabla z_h)_{L^2(\Omega)} \mbox{ for all } z_h\in Y_h.
\label{normal}
\end{equation}
\end{lemma}

An important tool in the numerical analysis is the construction of a discrete control $u_h^*\in U_{ad}^h$ which interpolates $\bar u$ in a certain sense, see Lemma \ref{L4.1} and Lemma \ref{L5.1}, and satisfies
\begin{equation}\label{eq:u_h}
(\nu\bar u-\partial_n\bar\varphi,u_h^*-\bar u)_{L^2(\Gamma)}=0.
\end{equation}
If the optimal control $\bar u\in H^s(\Gamma)$ with $s<1$ then we use a quasi-interpolant introduced by Casas and Raymond in \cite{Casas-Raymond2006}: Denote the boundary nodes of the mesh by $x_\Gamma^j$, $1\le j\le N(h)$, and
let $e_j$, $1\le j\le N(h)$,  be the nodal basis of $U_h$. We set
\begin{align*}
\bar d(x)&=\nu \bar u(x) -\partial_n \bar\varphi(x),\\
I_j&=\int_{x_\Gamma^{j-1}}^{x_\Gamma^{j+1}} \bar d(x)e_j(x)\,dx,
\end{align*}
and define a control $u_h^*=\sum_{j=1}^{N(h)}u_{h,j}^* e_j$ by its coefficients
\begin{equation}\label{eq:u*1}
u_{h,j}^*=\left\{
\begin{array}{ll}
\displaystyle{\frac{1}{I_j}\int_{x_\Gamma^{j-1}}^{x_\Gamma^{j+1}} \bar d(x)
	\bar u(x) e_j(x)\,d\sigma(x)} & \mbox{if } I_j\neq 0,\\ \\
\displaystyle{\frac{1}{h_{j-1}+h_j}\int_{x_\Gamma^{j-1}}^{x_\Gamma^{j+1}}
	\bar u(x)\,d\sigma(x)} & \mbox{if } I_j= 0.
\end{array}
\right.
\end{equation}
According to \cite[Lemma 7.5]{Casas-Raymond2006} the function $u_h^*$ belongs to $U_{ad}^h$. Moreover, it is constructed such that $u_h^*=\bar u$ on the active set, and it fulfills \eqref{eq:u_h}.

If $\bar u\in H^s(\Gamma)$ with $s\geq 1$, we use a modification of the standard Lagrange interpolant $I_h\bar u$ of $\bar u$, again denoted by $u_h^*\in U_{ad}^h$, which is defined by its coefficients as follows
\begin{equation}\label{eq:u*2}
	u_{h,j}^*=\begin{cases}
	a & \text{if }\min_{[x_\Gamma^{j-1},x_\Gamma^{j+1}]}\bar u(x)=a,\\
	b & \text{if }\max_{[x_\Gamma^{j-1},x_\Gamma^{j+1}]}\bar u(x)=b,\\
	\bar u(x_\Gamma^j)&\text{else}  ,
	\end{cases}
\end{equation}
cf. \cite[Section 2]{casas2003error}. Of course, if we consider control problems without control constraints, that is $-a=b=\infty$, the interpolant $u_h^*$ is just the Lagrange interpolant. In case of control bounds $a,b\in\mathbb{R}$, in order to get an unique definition of $u_h^*$, we need to assume that on each element only one control bound is active. However, due to the H\"older continuity of $\bar u$, which we have for $\bar u\in H^s(\Gamma)$ with $s\geq 1$, there exists a mesh size $h_0>0$ such that for all $h<h_0$ the above definition of the interpolant is unique.
Obviously, this interpolant belongs to $U_{ad}^h$. Moreover, it satisfies \eqref{eq:u_h} by construction. Indeed, whenever $\nu\bar u(x) -\partial_n\bar\varphi(x)\neq 0$, we have $u_h^*(x)-\bar u(x)=0$.

As already announced, we conclude this section by stating a general error estimate for the control and state errors which will serve as a basis for the subsequent error analysis.
\begin{theorem}\label{T3.2} For the solution of the continuous and the
discrete optimal control problem we have
\begin{align}
\|\bar u&-\bar u_h\|_{L^2(\Gamma)}+\|\bar y-\bar y_h\|_{L^2(\Omega)}\notag\\
&\le c\left(\|\bar u- u_h^*\|_{L^2(\Gamma)}+\|\bar y-S_h\bar u\|_{L^2(\Omega)}+\sup_{\psi_h\in U_h}\frac{\left|(\nabla \bar\varphi,\nabla S_h\psi_h)_{L^2(\Omega)}\right|}{\|\psi_h\|_{L^2(\Gamma)}}\right).\label{eq:general_estimate}
\end{align}
\end{theorem}
\begin{proof}
	First, let us define the intermediate error $e_h:=u_h^*-\bar u_h$. Then, we obtain
	\begin{align}
	  \|\bar u-\bar u_h\|_{L^2(\Gamma)}+\|\bar y-\bar y_h\|_{L^2(\Omega)}&\leq \|\bar u- u_h^*\|_{L^2(\Gamma)}+\| e_h\|_{L^2(\Gamma)}\notag \\&+\|\bar y-S_h u_h^*\|_{L^2(\Omega)}+\|S_he_h\|_{L^2(\Omega)}.\label{eq:generr1}
	\end{align}
To deal with the third term, we take into account the continuity of $S_h$:
\begin{align}
\|\bar y-S_h u_h^*\|_{L^2(\Omega)}&\leq
\|\bar y-S_h \bar u\|_{L^2(\Omega)}+\|S_h (\bar u- u_h^*)\|_{L^2(\Omega)}\notag\\
&\leq \|\bar y-S_h \bar u\|_{L^2(\Omega)}+ c\| \bar
u-u_h^*\|_{L^2(\Gamma)},\label{eq:stability}\end{align}
cf. \cite[Lemma 2.3 and Corollary 3.3]{Apel-Nicaise-Pfefferer2014}.
Accordingly, we only need estimates for the second and fourth term in
\eqref{eq:generr1}. We begin with estimating the second one, but as we will see
this also yields an estimate for the fourth term. There holds
\begin{align}
    \nu\| e_h\|_{L^2(\Gamma)}^2&= \nu(u_h^*-\bar u,e_h)_{L^2(\Gamma)}+\nu(\bar u-\bar u_h,e_h)_{L^2(\Gamma)}.\label{eq:generr2}
\end{align}
  Next, we consider the second term of~\eqref{eq:generr2} in detail. By adding the continuous and discrete variational inequalities \eqref{E3.1} and \eqref{VI}  with $u=\bar u_h\in U_{ad}$ and $u_h=u_h^*\in U_{ad}^h$, respectively, we deduce
  \[
    (\nu(\bar u_h-\bar u)+\partial_n\bar \varphi-\partial_n^h\bar\varphi_h,e_h)_{L^2(\Gamma)}+(\nu\bar u-\partial_n\bar\varphi,u_h^*-\bar u)_{L^2(\Gamma)}\geq 0.
  \]
	Rearranging terms and using~\eqref{eq:u_h} leads to
	\begin{equation*}
	  \nu(\bar u-\bar u_h,e_h)_{L^2(\Gamma)}\leq (\partial_n\bar \varphi-\partial_n^h\bar\varphi_h,e_h)_{L^2(\Gamma)}.
	\end{equation*}
	Integration by parts (cf.~\cite[Lemma 3.4]{Costabel1988}) {using $e_h=S_he_h$ on $\Gamma$},~\eqref{normal},~\eqref{E3.3} and~\eqref{discharmonic} yield
	\begin{align}
	  \nu(\bar u-\bar u_h,e_h)_{L^2(\Gamma)}&\leq (\Delta \bar \varphi+(\bar y_h-y_\Omega),S_he_h)_{L^2(\Omega)}+(\nabla (\bar\varphi-\bar \varphi_h),\nabla S_he_h)_{L^2(\Omega)}\notag\\
	  &= (\bar y_h -\bar y,S_he_h)_{L^2(\Omega)}+(\nabla \bar\varphi,\nabla S_he_h)_{L^2(\Omega)}\notag\\
	  &= (S_h u_h^* -\bar y,S_he_h)_{L^2(\Omega)}-\|S_he_h\|_{L^2(\Omega)}^2+(\nabla \bar\varphi,\nabla S_he_h)_{L^2(\Omega)}\label{eq:generr3}.
	\end{align}
	By collecting the estimates~\eqref{eq:generr2} and~\eqref{eq:generr3} we obtain
	\begin{align}
	  \nu\| e_h\|_{L^2(\Gamma)}^2&+\|S_he_h\|_{L^2(\Omega)}^2\notag\\
	  \leq&\quad \nu(u_h^*-\bar u,e_h)_{L^2(\Gamma)}+(S_h u_h^* -\bar y,S_he_h)_{L^2(\Omega)}+(\nabla \bar\varphi,\nabla S_he_h)_{L^2(\Omega)}\notag\\
	  \leq &\quad\nu\|u_h^*-\bar u\|_{L^2(\Gamma)}\|e_h\|_{L^2(\Gamma)}+\|S_h u_h^* -\bar y\|_{L^2(\Omega)}\|S_he_h\|_{L^2(\Omega)}\notag\\
&+\sup_{\psi_h\in U_h}\frac{\left|(\nabla \bar\varphi,\nabla S_h\psi_h)_{L^2(\Omega)}\right|}{\|\psi_h\|_{L^2(\Gamma)}}\|e_h\|_{L^2(\Gamma)}.\label{eq:generr5}
	\end{align}
	From the Young inequality we can deduce
	\begin{align}
	  &\| e_h\|_{L^2(\Gamma)}+\|S_he_h\|_{L^2(\Omega)}\notag\\
	  &\leq c\left(\|u_h^*-\bar u\|_{L^2(\Gamma)}+\|S_h u_h^* -\bar y\|_{L^2(\Omega)}+\sup_{\psi_h\in U_h}\frac{\left|(\nabla \bar\varphi,\nabla S_h\psi_h)_{L^2(\Omega)}\right|}{\|\psi_h\|_{L^2(\Gamma)}}\right).\label{eq:generr4}
	\end{align}
	Finally, the assertion is a consequence from~\eqref{eq:generr1}, \eqref{eq:stability} and~\eqref{eq:generr4}.
\end{proof}

\section{Problems without control constraints}\label{sec:4}
In the rest of the paper, we will always assume that $\{\mathcal{T}_h\}$ is a
quasi-uniform family of meshes.
However, if the underlying mesh has a
certain structure then it is possible to improve the error estimates. These
special quasi-uniform
meshes are called superconvergence meshes or $O(h^2)$-irregular meshes; for the precise definition we refer to Definition
\ref{def:superconvergence}.
The main result of this section is the following one.
\begin{theorem}\label{main:unc}Suppose that either $\lambda < 1$
and $y_\Omega\in
{L^2(\Omega)}$, or $y_\Omega\in
W^{1,p^*}(\Omega)$ for some
$p^*>2$.
Then it holds
\begin{align}\label{T41q}
\|\bar u-\bar u_h\|_{L^2(\Gamma)}+\|\bar y-\bar y_h\|_{L^2(\Omega)} & \le ch^{s}
|\log h|^r\quad \notag \\ & \forall s\in \mathbb{R} \text{ such that }s<\lambda-\tfrac12\text{ and }s\leq 1,
\end{align}
where $r$ is equal to one for $\lambda-\tfrac12\in (1,\tfrac32]$ and equal to zero else.
If, further, $\{\mathcal{T}_h\}$ is $O(h^2)$-irregular 
according to
Definition \ref{def:superconvergence}, then
\begin{align}\label{T41s}
\|\bar u-\bar u_h\|_{L^2(\Gamma)}+\|\bar y-\bar y_h\|_{L^2(\Omega)} \le ch^{s}\quad \forall s<\min\{\tfrac32,\lambda-\tfrac12\}.
\end{align}
\end{theorem}
For the proof, we are going to estimate the three terms that appear in the general estimate of Theorem \ref{T3.2}. Whereas the first two terms in \eqref{eq:general_estimate} can be estimated by standard techniques, the third one needs special care. Analogously to the derivation of \eqref{eq:generr3}, this term can formally be rewritten as
\begin{align*}
	\sup_{\psi_h\in U_h}\frac{\left|(\nabla \bar\varphi,\nabla S_h\psi_h)_{L^2(\Omega)}\right|}{\|\psi_h\|_{L^2(\Gamma)}}&=\sup_{\psi_h\in U_h}\frac{\left|(\nabla (\bar\varphi-R_h\bar\varphi),\nabla S_h\psi_h)_{L^2(\Omega)}\right|}{\|\psi_h\|_{L^2(\Gamma)}}\\
	&=\sup_{\psi_h\in U_h}\frac{\left|(\partial_n\bar
	\varphi-\partial_n^hR_h\bar\varphi,\psi_h)_{L^2(\Gamma)}\right|}{\|\psi_h\|_{L^2(\Gamma)}},
\end{align*}
where $\partial_n^hR_h\bar\varphi$ is defined as in \eqref{normal} just by
replacing $\bar y_h$ with $\bar y$ and $\bar \varphi_h$ with the
Ritz-projection $R_h\bar\varphi$ of $\bar\varphi$ on $Y_{0,h}$. Thus, we are
interested in the error between the normal derivative of the adjoint state and
the corresponding discrete normal derivative of its Ritz-projection. In order
to estimate the above term, we will pursue two different strategies. The first
one relies on local and global $W^{1,\infty}$-discretization error estimates. In
case of general quasi-uniform meshes, this will result in a convergence order
of
$O(h^{s}|\log h|^r)$ for all $s\in \mathbb{R}$ such that $s<\lambda-\tfrac12$ and $s\leq 1$,
where $r$ is equal to one for $\lambda-\tfrac12\in (1,\tfrac32]$ and equal to zero else.
The second strategy will rely on special super-convergence meshes as introduced in \cite{bank2003}. The idea to use such meshes in the context of Dirichlet boundary control problems originally stems from
\cite{Deckelnick-Gunther-Hinze2009}. In contrast to the setting in that
reference, we are not concerned with smoothly bounded domains but with
polygonal domains. For that reason we need to extend the corresponding
estimates to that case, that is, we have to deal with less regular functions
due to the appearance of corner singularities. This will yield an approximation
rate of
$O(h^s)$ with $s<\min\{\tfrac32,\lambda-\tfrac12\}$, which results in an improvement for
domains with interior angles less than $2\pi/3$.

\begin{lemma}\label{L4.1}Suppose $y_\Omega\in H^{t}(\Omega)
{\cap
L^2(\Omega)}$ for all
$t<
{\min\{1,\lambda-1\}}$. Then we have
\[
\|\bar u- u_h^*\|_{L^2(\Gamma)}+ \|\bar y -S_h\bar u\|_{L^2(\Omega)}\le c h^s\quad \forall s<\min\{\tfrac32,\lambda-\tfrac12\}.
\]
\end{lemma}
\begin{proof}
We know from Lemma \ref{L2.3} that the control satisfies $\bar u\in
H^{s}(\Gamma)$ for all $s<\min\{\tfrac32,\lambda-\tfrac12\}$.
If $s<1$, we choose $u_h^*$ as defined in \eqref{eq:u*1}, and the estimate for the control follows from \cite[Eq. (7.10)]{Casas-Raymond2006} by setting $s=1-\tfrac1p$ with $p\in(1,\infty)$. If $1\leq s <\tfrac32$, we have $\bar u\in H^s(\Gamma) \hookrightarrow C^{0,s-\tfrac12}(\Gamma)$ due to the Sobolev embedding theorem. Thus, the modified Lagrange interpolant $u_h^*$ from \eqref{eq:u*2} is well-defined. Actually, in the present case, $u_h^*$ is just the Lagrange interpolant. As a consequence, the error estimate for the control is given by a standard estimate for the Lagrange interpolant.

Again from Lemma \ref{L2.3}, the optimal state satisfies $\bar y\in
H^{s+\tfrac12}(\Omega)$, for all $s<\min\{\tfrac32,\lambda-\tfrac12\}$. Thus,
$\|\bar y -S_h\bar u\|_{H^1(\Omega)}\le ch^{s-\tfrac12}$ for all $s<\min\{\tfrac32,\lambda-\tfrac12\}$ if $\lambda\ge 1$. By the Aubin--Nitsche method we obtain
\begin{equation}\label{fem:unc}
\|\bar y -S_h\bar u\|_{L^2(\Omega)}\le ch^{s-\tfrac12+{\min\{1,
{s+\tfrac12}\}}}\quad
\forall s<\min\{\tfrac32,\lambda-\tfrac12\},
\end{equation}
cf. for instance \cite{Bartels2004}. Since $s+\tfrac12$ can be chosen greater than $\tfrac12$, we have the desired result in case that $\lambda\geq 1$. For $\lambda < 1$ we do not have  $\bar y \in
H^1(\Omega)$ such that standard techniques for estimating finite element errors fail. However, in this
case we can directly refer to Remark 5.4 of \cite{Apel-Nicaise-Pfefferer2014}.
\end{proof}

\begin{lemma}\label{L4.3} Suppose that either $\lambda < 1$ and $y_\Omega\in
{L^2(\Omega)}$, or $y_\Omega\in
W^{1,p^*}(\Omega)$ for some
$p^*>2$.
Then there is the estimate
\[
\sup_{\psi_h\in U_h}\frac{\left|(\nabla \bar\varphi,\nabla
S_h\psi_h)_{L^2(\Omega)}\right|}{\|\psi_h\|_{L^2(\Gamma)}}\le c
{h^s|\log h|^r\quad \forall s\in \mathbb{R} \text{ such that }s<\lambda-\tfrac12\text{ and }s\leq 1,}
\]
where $r$ is equal to one for $\lambda-\tfrac12\in (1,\tfrac32]$ and equal to zero else.
\end{lemma}
\begin{proof}
As above, we denote by $R_h$ the operator that maps a function of $H^1_0(\Omega)$ to its Ritz-projection in $Y_{0,h}$. In addition, we introduce the extension operator $\tilde S_h$ which extends a function belonging to $U_h$ to one in $Y_h$ by zero. Using the norm equivalence in finite dimensional spaces on a reference domain we easily infer for any $\psi_h\in U_h$ and $q\in[1,\infty]$
\begin{equation}\label{eq:local1}
{\|\tilde S_h \psi_h\|_{L^q(\Omega)}+h\|\nabla\tilde S_h
	\psi_h\|_{L^q(\Omega)}\leq c h^{1/q}\|\psi_h\|_{L^q(\Gamma)}.}
\end{equation}
Since $S_h\psi_h$ is discrete harmonic, we obtain together with the orthogonality properties of the Ritz-projection the identity
\begin{align}
(\nabla \bar\varphi,\nabla S_h\psi_h)_{L^2(\Omega)}&=(\nabla (\bar\varphi-R_h\bar \varphi),\nabla S_h\psi_h)_{L^2(\Omega)}
=(\nabla (\bar\varphi-R_h\bar \varphi),\nabla \tilde S_h\psi_h)_{L^2(\Omega)},\label{eq:local3}
\end{align}
where we employed that $(S_h-\tilde S_h)\psi_h$ belongs to $Y_{0,h}$.

Now, we distinguish the three cases $\omega_i<\pi/2$, $\omega_i<\pi$ and
$\omega_i<2\pi$ for $i=1,\ldots,M$.

In the first one, we know from Lemma \ref{C4.1r}
that the optimal adjoint state belongs to $W^{3,q}(\Omega)$ (for some $q>2$), which is continuously embedded in $W^{2,\infty}(\Omega)$. Consequently, a global $W^{1,\infty}$-discretization error estimate from e.g. \cite{rannacher1982,demlow2012}, and \eqref{eq:local1} yield
\begin{equation}\label{eq:local4}
(\nabla (\bar\varphi-R_h\bar \varphi),\nabla \tilde S_h\psi_h)_{L^2(\Omega)}\leq \|\nabla (\bar\varphi-R_h\bar \varphi)\|_{L^\infty(\Omega)}\|\nabla \tilde S_h\psi_h\|_{L^1(\Omega)}\leq ch\|\psi_h\|_{L^1(\Gamma)},
\end{equation}
which represents, together with \eqref{eq:local3} and the embedding $L^2(\Gamma)\hookrightarrow L^1(\Gamma)$, the desired result for $\omega_i<\pi/2$, $i=1,\ldots,M$.

Next, we consider the case
$\omega_i<\pi$ for $i=1,\ldots,M$. For simplicity, we assume that the
domain has only one corner with an interior angle greater or equal to $\pi/2$.
However, the proof extends to the general case in a natural way. In the
following, that corner is located at the origin. Furthermore, we denote its
interior angle by $
{\omega_1}$,
the distance to that corner by $r_1$,
and the corresponding
leading singular exponent by ${\lambda_1=\pi/\omega_1}$. According to
Lemma \ref{C4.1r}, the optimal adjoint state admits the splitting
\begin{equation}\label{eq:local2}
\bar\varphi=\bar{\varphi}_{\reg}+\bar\varphi_{\sing},
\end{equation}
where $\bar{\varphi}_{\reg}$ belongs to $W^{3,q}(\Omega)$ with some $q>2$. Combining \eqref{eq:local3} and \eqref{eq:local2} yields the identity
\begin{align}
	(\nabla (\bar\varphi-R_h\bar \varphi),\nabla \tilde S_h\psi_h)_{L^2(\Omega)}&=(\nabla (\bar\varphi_{\sing}-R_h\bar \varphi_{\sing}),\nabla \tilde S_h\psi_h)_{L^2(\Omega)}\notag\\
	&\quad+(\nabla (\bar\varphi_{\reg}-R_h\bar \varphi_{\reg}),\nabla \tilde S_h\psi_h)_{L^2(\Omega)}.\label{eq:local13}
\end{align}
For the latter term, we can argue as in \eqref{eq:local4} to show first order convergence, i.e.,
\begin{equation}\label{eq:local12}
(\nabla (\bar\varphi_{\reg}-R_h\bar \varphi_{\reg}),\nabla \tilde S_h\psi_h)_{L^2(\Omega)}\leq ch\|\psi_h\|_{L^1(\Gamma)}\leq ch\|\psi_h\|_{L^2(\Gamma)}.
\end{equation}
In order to estimate the singular term, we decompose the neighborhood of the critical corner in subdomains $\Omega_J$ which are defined by
\begin{align*}
\Omega_I&:=\{x:|x|\leq d_I\}\quad\text{and}\quad
\Omega_J:=\{x:d_{J+1}\leq |x|\leq d_J\}\quad \text{for }J=I-1,\ldots,1.
\end{align*}
We set the radii $d_J$ equal to $2^{-J}$ and choose the index $I$ in such a way that $d_I=2^{-I}=c_Ih$ with a constant $c_I$. Below, this constant is chosen large enough such that on the one hand local $W^{1,\infty}$-finite element error estimates from \cite[Corollary 1]{demlow2012} are applicable on the strips $\Omega_J$, see \eqref{eq:local6}, and on the other hand the validity of the weighted error estimate \eqref{eq:local10} is guaranteed. Moreover, we set
\[
	\Omega_0:=\Omega\backslash\Omega_R \quad \text{with} \quad \Omega_R:=\bigcup_{J=1}^I\Omega_J
\]
and
\[
	\Omega_J':=\Omega_{J-1}\cup\Omega_{J}\cup\Omega_{J+1}
\]
with the obvious modifications for $J=0$ and $J=I$. Using this kind of covering, we obtain
\begin{align}
	(\nabla (\bar\varphi_{\sing}-R_h\bar \varphi_{\sing}),&\nabla \tilde S_h\psi_h)_{L^2(\Omega)}=\sum_{J=0}^I(\nabla (\bar\varphi_{\sing}-R_h\bar \varphi_{\sing}),\nabla \tilde S_h\psi_h)_{L^2(\Omega_J)}\notag\\
	&\leq \sum_{J=0}^I\|\nabla (\bar\varphi_{\sing}-R_h\bar \varphi_{\sing})\|_{L^\infty(\Omega_J)}\|\nabla \tilde S_h\psi_h\|_{L^1(\Omega_J)}.\label{eq:local5}
\end{align}
Arguing as in \eqref{eq:local1}, we get
\begin{equation}\label{eq:local7}
 \|\nabla \tilde S_h\psi_h\|_{L^1(\Omega_J)}\leq c\|\psi_h\|_{L^1(\partial\Omega_J'\cap \Gamma)}.
\end{equation}
Having chosen the constant $c_I$ large enough, local $W^{1,\infty}$-error estimates from \cite[Corollary 1]{demlow2012} yield
\begin{align}
&\|\nabla (\bar\varphi_{\sing}-R_h\bar \varphi_{\sing})\|_{L^\infty(\Omega_J)}\notag\\ &\leq c\left(\|\nabla(\bar\varphi_{\sing}-I_h\bar{\varphi}_{\sing})\|_{L^\infty(\Omega_J')}+d_J^{-2}\|\bar\varphi_{\sing}-R_h\bar\varphi_{\sing}\|_{L^2(\Omega_J')}\right),\label{eq:local6}
\end{align}
where $I_h\bar{\varphi}$ denotes the Lagrange interpolant of $\bar{\varphi}$.
Notice, according to \cite[Remark 2]{demlow2012}, this inequality is only valid for any $J=0,\ldots,I-2$ if the domain $\Omega$ is non-convex, i.e. $\omega_1>\pi$. Now, let $\sigma:=r_1+d_I$, which possesses the properties $\sigma\sim d_J$ for $x\in \bar\Omega_J'$ and $\min_{x\in\Omega}\sigma\sim h$. By combining \eqref{eq:local5}--\eqref{eq:local6}, we infer
\begin{align}
	&(\nabla (\bar\varphi_{\sing}-R_h\bar \varphi_{\sing}),\nabla \tilde S_h\psi_h)_{L^2(\Omega)}\notag\\
	&\leq c\sum_{J=0}^I\left(\|\nabla(\bar\varphi_{\sing}-I_h\bar{\varphi}_{\sing})\|_{L^\infty(\Omega_J')}+d_J^{-2}\|\bar\varphi_{\sing}-R_h\bar\varphi_{\sing}\|_{L^2(\Omega_J')}\right)\|\psi_h\|_{L^1(\partial\Omega_J'\cap \Gamma)}\notag\\
	&\leq c\left(\!\|\sigma^{1/2}\nabla(\bar\varphi_{\sing}\!-\!I_h\bar{\varphi}_{\sing})\|_{L^\infty(\Omega)}\!+\!h^{-1}\|\sigma^{-1/2}(\bar\varphi_{\sing}\!-\!R_h\bar\varphi_{\sing})\|_{L^2(\Omega)}\!\right)\!\|\sigma^{-1/2}\psi_h\|_{L^1(\Gamma)}.\label{eq:local8}
\end{align}
The second derivatives of the singular part $\bar{\varphi}_{\sing}$ behave like
$r^{{\lambda}-2}$ {for $\lambda\not=2$ and} like $\log r$ if ${\lambda}=2$, cf.
Lemma \ref{C4.1r}. Thus, by using standard interpolation error estimates (on
the strips $\Omega_J$), we get
for $\lambda_1>1$, hence for
$\omega_1<\pi$,
\begin{equation}\label{eq:local9}
	\|\sigma^{1/2}\nabla(\bar\varphi_{\sing}-I_h\bar{\varphi}_{\sing})\|_{L^\infty(\Omega)}\leq
	 c h^{\min\{1,\lambda-1/2\}}.
\end{equation}
From \cite[Corollary 3.62]{pfefferer2014} {(setting $\tau=\frac12$ and $\gamma=2-\lambda$ there)} we know that for $c_I$ large enough there holds
\begin{equation}\label{eq:local10}
	\|\sigma^{-1/2}(\bar\varphi_{\sing}-R_h\bar\varphi_{\sing})\|_{L^2(\Omega)}\leq
	 c h^{\min\{2,\lambda+1/2\}}|\log h|^{1/2}.
\end{equation}
Notice that in that reference problems with Neumann boundary conditions are considered. However, the proof for the present problem is just a word by word repetition. Next, the Cauchy-Schwarz inequality and basic integration yield
\begin{equation}\label{eq:local11}
	\|\sigma^{-1/2}\psi_h\|_{L^1(\Gamma)}\leq \|\sigma^{-1/2}\|_{L^2(\Gamma)}\|\psi_h\|_{L^2(\Gamma)}\leq c|\log h|^{1/2}\|\psi_h\|_{L^2(\Gamma)}.
\end{equation}
By collecting the results from \eqref{eq:local8}--\eqref{eq:local11}, we obtain
\begin{equation}\label{eq:local14}
	(\nabla (\bar\varphi_{\sing}-R_h\bar \varphi_{\sing}),\nabla \tilde
	S_h\psi_h)_{L^2(\Omega)}\leq c  h^{\min\{1,\lambda-1/2\}}|\log
	h|\,\|\psi_h\|_{L^2(\Gamma)},
\end{equation}
which yields together with \eqref{eq:local12}, \eqref{eq:local13} and \eqref{eq:local3} the assertion in the second case.

Finally, we consider the case $\omega_i<2\pi$ for $i=1,\ldots,M$. Similar
to the foregoing considerations, we assume that only the angle $\omega_1$ is
greater or equal to $\pi$ and hence $1/2<\lambda_1\leq1$. According to
\eqref{eq:local3}, the Cauchy-Schwarz inequality, \eqref{eq:local1},
and a standard finite element error estimate, we obtain
\begin{align*}
(\nabla \bar\varphi,\nabla S_h\psi_h)_{L^2(\Omega)}&=(\nabla (\bar\varphi-R_h\bar \varphi),\nabla \tilde S_h\psi_h)_{L^2(\Omega)}\\&\leq \|\nabla (\bar\varphi-R_h\bar \varphi)\|_{L^2(\Omega)}\|\nabla \tilde S_h\psi_h\|_{L^2(\Omega)}\\
&\leq c h^{
{s}}\|\psi_h\|_{L^2(\Gamma)}
\end{align*}
for all $s<\lambda_1-1/2$. This ends the proof.
\end{proof}

\begin{remark}
Related results to those of Lemma \ref{L4.3}, which are
established by using similar techniques, can be found in
\cite{apel2015finite,horger2013optimal,melenk2012quasi,pfefferer2014}.
\end{remark}

According to the previous lemma, the critical term in the general estimate \eqref{eq:general_estimate} converges with an order close to one provided that the interior angles are less $2\pi/3$. However, it is possible to improve the convergence rate if we assume a certain structure of the underlying mesh. The following definition for superconvergence meshes can be found in \cite{bank2003}. Those have been used in \cite{Deckelnick-Gunther-Hinze2009} in the context of Dirichlet boundary control problems in the case of smoothly bounded domains.
\begin{definition}\label{def:superconvergence}
	The triangulation $\mathcal{T}_h$ is called to be $O(h^{2\sigma})$-irregular if the following conditions hold:
	\begin{enumerate}
		\item The set of interior edges $\mathcal{E}$ of the triangulation $\mathcal{T}_h$ is decomposed into two disjoint sets $\mathcal{E}_1$ and $\mathcal{E}_2$ which fulfill the following properties:
		\begin{itemize}
			\item For each $e\in\mathcal{E}_1$, let $T$ and $T'$ denote the two elements of the triangulation $\mathcal{T}_h$ that share this edge $e$. Then the lengths of any two opposite edges of the quadrilateral $T\cup T'$ differ only by $O(h^2)$.
			\item $\sum_{e\in\mathcal{E}_2}(|T|+|T'|)=O(h^{2\sigma})$.
		\end{itemize}
		\item The set of the boundary vertexes $\mathcal{P}$ is decomposed into two disjoint set $\mathcal{P}_1$ and $\mathcal{P}_2$ which satisfy the following properties:
		\begin{itemize}
			\item For each vertex $x\in\mathcal{P}_1$, let $e$ and $e'$ be the two boundary edges sharing this vertex as an endpoint. Denote by $T$ and $T'$ the elements having $e$ and $e'$, respectively, as edges and let $t$ and $t'$ be the corresponding unit tangents. Furthermore, take $e$ and $e'$ as one pair of corresponding edges, and make a clockwise traversal of $\partial T$ and $\partial T'$ to define two additional corresponding edge pairs. Then $|t-t'|=O(h)$ and the lengths of any two corresponding edges only differ by $O(h^2)$.
			\item $|\mathcal{P}_2|=c$ with a constant $c$ independent of $h$.
		\end{itemize}
	\end{enumerate}
\end{definition}

Next, let us recall {a} result from \cite[Lemma 5.2]{Deckelnick-Gunther-Hinze2009}, which leads us to Lemma~\ref{L4.4}.

\begin{lemma}\label{lemma:DGH}
Let $\Omega_h$ be any polygonal domain with boundary $\Gamma_h$. Suppose that the triangulation $\mathcal{T}_h$ of $\Omega_h$ is $O(h^{2\sigma})$ irregular and let $f\in W^{3,q}(\Omega_h)$ for some $q>2$. Then for any $\phi_h\in Y_h$ there holds
\[
	\left|\int_{\Omega_h} \nabla(f-I_hf)\cdot\nabla \phi_h\right|\leq c\|f\|_{W^{3,q}(\Omega_h)}\left(h^{1+\min\{1,\sigma\}}\|\phi_h\|_{H^1(\Omega_h)}+h^{3/2}\|\phi_h\|_{L^2(\Gamma_h)}\right),
\]
where $I_h f\in Y_h$ denotes the piecewise linear Lagrange interpolant.
\end{lemma}
\begin{lemma}\label{L4.4}
Suppose that either $\lambda < 1$ and $y_\Omega\in
	L^2(\Omega)$, or $y_\Omega\in
	W^{1,p^*}(\Omega)$ for some
	$p^*>2$. Suppose further that $\{\mathcal{T}_h\}$ is a family of
$O(h^2)$-irregular meshes. Then it holds
\[
\sup_{\psi_h\in U_h}\frac{\left|(\nabla \bar\varphi,\nabla S_h\psi_h)_{L^2(\Omega)}\right|}{\|\psi_h\|_{L^2(\Gamma)}}\le
{c h^s\quad \forall s\in\mathbb{R}\text{ such that }s<\lambda-\frac12\text{ and }s\leq\frac32.}
\]
\end{lemma}

\begin{proof}
First we observe that
\[
	\left|(\nabla \bar\varphi,\nabla S_h\psi_h)_{L^2(\Omega)}\right|=\left|(\nabla (\bar\varphi-I_h\bar{\varphi}),\nabla S_h\psi_h)_{L^2(\Omega)}\right|
\]
since $S_h$ represents the discrete harmonic extension operator and $\bar \varphi$ has zero boundary conditions.
If at least one interior angle $\omega_i$ is greater or equal to $2\pi/3$, we have
$\lambda\leq3/2$ and therefore $\lambda-1/2\leq 1$. Consequently, there is no
advantage in taking a superconvergence mesh and we can apply the result for
quasi-uniform meshes.
If $\omega_i<\pi/2$ for $i=1,\ldots,M$, and hence
$\lambda>2$, we can directly apply the results of Lemma \ref{lemma:DGH} since
$\bar\varphi \in W^{3,q}(\Omega)$ for some $q>2$ according to Lemma
\ref{C4.1r}.
For these reasons, we focus in the following only on the case $3/2<\lambda\leq
2$.
We are in this case if the largest interior angle, denoted by
$\omega_1$ in the following, fulfills
$\pi/2\leq \omega_1 < 2\pi/3$. For simplicity, we assume as in the proof of
Lemma \ref{L4.3} that the remaining angles are less than $\pi/2$. However, the
proof again extends to the general case in a natural way.
According to Lemma \ref{C4.1r} we have that
\[
	\bar\varphi=\bar\varphi_{\sing}+\bar{\varphi}_{\reg},
\]
where $\bar{\varphi}_{\reg}$ belongs to $W^{3,q}(\Omega)$ with some $q>2$. For
the regular part we can again employ Lemma \ref{lemma:DGH} to obtain the
order $3/2$. The singular part behaves at worst like $r_1^{\lambda}$ or
like $r_1^2|\log r_1|$, respectively, if $\lambda=2$.
As before, we would like to use Lemma \ref{lemma:DGH} to get the corresponding
estimate. For that purpose, we decompose the domain into two disjoint subsets
$\Omega_{h,1}$ and $\Omega_{h,2}$. The set $\Omega_{h,1}$ consists of the
elements of the triangulation which have contact to the corner
$x_1$, while $\Omega_{h,2}:=\Omega\backslash\Omega_{h,1}$. Since the
triangulation of $\Omega$ is $O(h^2)$ irregular, the triangulation of
$\Omega_{h,2}$ is $O(h^2)$ irregular, either. Applying Lemma \ref{lemma:DGH}
yields for any $q>2$
\begin{align}
	\left|\int_{\Omega_{h,2}} \nabla(\bar\varphi_{\sing}-I_h\bar\varphi_{\sing})\cdot\nabla S_h\psi_h\right|\leq c\|\bar\varphi_{\sing}\|_{W^{3,q}(\Omega_{h,2})}\left(h^{2}\|S_h\psi_h\|_{H^1(\Omega_{h,2})}\right.\notag\\
\quad\left.+h^{3/2}\left(\|S_h\psi_h\|_{L^2(\partial\Omega_{h,2}\cap\Omega)}+\|\psi_h\|_{L^2(\partial\Omega_{h,2}\cap\Gamma)}\right)\right).\label{eq:Omegah21}
\end{align}
Since the number of elements in $\Omega_{h,1}$ is bounded independently of $h$ and $\partial \Omega_{h,1}\cap\Omega = \partial \Omega_{h,2}\cap\Omega$, we have that $\left|\partial\Omega_{h,2}\cap\Omega\right|\sim h$.
Using this fact, the H\"older inequality,
 and a discrete Sobolev inequality, we obtain
\begin{equation}\label{eq:Omegah22}
	\|S_h\psi_h\|_{L^2(\partial\Omega_{h,2}\cap\Omega)}\leq ch^{1/2}\|S_h\psi_h\|_{L^\infty(\Omega)}\leq ch^{1/2}\left|\log h\right|^{1/2}\|S_h\psi_h\|_{H^1(\Omega)}.
\end{equation}
Define $\tilde S_h$ as the zero extension operator as in the proof of Lemma
\ref{L4.3}. Since $S_h\psi_h$ denotes the discrete harmonic extension of
$\psi_h$, we infer
\[
\|\nabla(S_h-\tilde
S_h)\psi_h\|_{L^2(\Omega)}\leq \|\nabla\tilde S_h\psi_h\|_{L^2(\Omega)}.
\]
Using this in combination with the Poincar\'{e} inequality yields
\begin{align}
	\|S_h\psi_h\|_{H^1(\Omega)}&\leq \|(S_h-\tilde
	S_h)\psi_h\|_{H^1(\Omega)}+\|\tilde
	S_h\psi_h\|_{L^2(\Omega)}+\|\nabla\tilde S_h\psi_h\|_{L^2(\Omega)}\notag\\
	&\leq c\|\nabla(S_h-\tilde
	S_h)\psi_h\|_{L^2(\Omega)}+\|\tilde
	S_h\psi_h\|_{L^2(\Omega)}+\|\nabla\tilde S_h\psi_h\|_{L^2(\Omega)}\notag\\
	&\leq\|\tilde
	S_h\psi_h\|_{L^2(\Omega)}+c\|\nabla\tilde S_h\psi_h\|_{L^2(\Omega)}\notag\\
	&\leq ch^{-1/2} \|\psi_h\|_{L^2(\Gamma)},
\end{align}
where we used \eqref{eq:local1} in the last step.

Next, we observe that the third derivatives of $\bar{\varphi}_{\sing}$ behave like $r_1^{{\lambda}-3}$ such that we can conclude for some arbitrary $\varepsilon>0$ (depending on $q$)
\begin{equation}\label{eq:Omegah23}
	\|\bar\varphi_{\sing}\|_{W^{3,q}(\Omega_{h,2})}\leq c\|r_1^{{\lambda}-3}\|_{L^q(\Omega_{h,2})}\leq c h^{\lambda-2-\varepsilon}
\end{equation}
since $\min_{x\in\Omega_{h,2}}r_1(x)\sim h$. Collecting
\eqref{eq:Omegah21}--\eqref{eq:Omegah23} yields
\[
	\left|\int_{\Omega_{h,2}} \nabla(\bar\varphi_{\sing}-I_h\bar\varphi_{\sing})\cdot\nabla S_h\psi_h\right|\leq c h^{\lambda-1/2-\epsilon}\|\psi_h\|_{L^2(\Gamma)},
\]
which represents the desired result for the subdomain $\Omega_{h,2}$. Finally, for the subdomain $\Omega_{h,1}$, we conclude by inserting a standard interpolation error estimate and the a priori estimate for the operator $S_h$ as before that
\begin{align*}
	\left|\int_{\Omega_{h,1}} \nabla(\bar\varphi_{\sing}-I_h\bar\varphi_{\sing})\cdot\nabla S_h\psi_h\right|&\leq c \|\nabla(\bar\varphi_{\sing}-I_h\bar\varphi_{\sing})\|_{L^2(\Omega_{h,1})}\|\nabla S_h\psi_h\|_{L^2(\Omega_{h,1})}\\
	&\leq c h^{1/2}|\bar{\varphi}_{\sing}|_{H^2(\Omega_{h,1})}\|\psi_h\|_{L^2(\Gamma)}.
\end{align*}
After observing that the second derivatives of $\bar{\varphi}_{\sing}$ behave like $r_1^{{\lambda}-2}$ or $\log r_1$, respectively, if ${\lambda}=2$, and that $\max_{x\in\Omega_{h,1}} r_1(x)\sim h$, we get the desired result for the subdomain $\Omega_{h,1}$.
\end{proof}

\begin{proof}[Proof of Theorem \ref{main:unc}]
The result is obtained from the general error estimate in Theorem \ref{T3.2} using the estimates in Lemmata \ref{L4.1}, \ref{L4.3} and \ref{L4.4}.
\end{proof}
\begin{remark}\label{R-unc}
  As we commented in the introduction,
  it is possible, though unlikely, that the coefficient $c_{j,1}$ of the leading singular exponent vanishes. In this case, we can replace the parameter $\lambda$ in Theorem \ref{main:unc} by $\min\{\Lambda_j\}$.
\end{remark}

\section{The control constrained case}\label{sec:5}
This section is devoted to the numerical analysis of control constrained Dirichlet control problems.
As we will see, the convergence rates in convex domains coincide with those for the unconstrained problems. More precisely, we will prove the following theorem.
\begin{theorem}\label{main:cc}Suppose that either $\lambda < 1$
	and $y_\Omega\in
{L^2(\Omega)}$, or $y_\Omega\in
	W^{1,p^*}(\Omega)$ for some
	$p^*>2$.
Moreover, assume that the optimal control has a finite number of kink points.
	Then it holds
	\begin{align*}
	\|\bar u-\bar u_h\|_{L^2(\Gamma)}+\|\bar y-\bar y_h\|_{L^2(\Omega)}& \le ch^{s}
|\log h|^r\quad \\ &\forall s\in \mathbb{R} \text{ such that }s<\lambda-\tfrac12\text{ and }s\leq 1,
	\end{align*}
where $r$ is equal to one for $\lambda-\tfrac12\in (1,\tfrac32]$ and equal to zero else.
	If, further, $\{\mathcal{T}_h\}$ is $O(h^2)$-irregular according to
		Definition \ref{def:superconvergence}, then
	\begin{align*}
	\|\bar u-\bar u_h\|_{L^2(\Gamma)}+\|\bar y-\bar y_h\|_{L^2(\Omega)} \le ch^{s}\quad \forall s<\min\{\tfrac32,\lambda-\tfrac12\}.
	\end{align*}
\end{theorem}
The proof of this theorem is postponed to Section \ref{sec:proof1}. As already observed, this is exactly the result which we have proven in the unconstrained case. However, if the underlying domain is non-convex, the approximation rates in the control constrained case can be improved. In this regard, one of our results relies on a structural assumption on the discrete optimal control which we formulate next. Through this section we will shortly write
\[\mathbb{H} = \{j: \lambda_j<1\mbox{ and }c_{j,1}\neq 0\}.\]
\begin{assumption}\label{assump1}
	There exists some $h_0>0$ such that for every $j\in\mathbb{H}$,
	there exists $\tilde \rho_j>0$ independent of $h$ such that $\bar u_h(x)=\bar u(x)$ for all $h<h_0$ if $|x-x_j|<\tilde \rho_j$.
\end{assumption}
Let us comment on Assumption \ref{assump1}. In Lemma \ref{L2.5} it was established  that in the neighbourhood of a non-convex corner, the optimal control will normally be constant and either equal to the lower or the upper bound. Assumption \ref{assump1} says that this property is inherited by the discrete optimal control.
	
One of our main results in the constrained case is now given as follows.
\begin{theorem}\label{T5main} Suppose $y_\Omega\in W^{1,p^*}(\Omega)$ for some $p^*>2$. Moreover, let either $\lambda >1$ or Assumption \ref{assump1} be satisfied,
and assume that the optimal control has a finite number of kink points.
Then there is the estimate
\begin{align}\label{T51q}
\|\bar u-\bar u_h\|_{L^2(\Gamma)}+\|\bar y-\bar y_h\|_{L^2(\Omega)}&\le ch^{s}
|\log h|^r\notag \\&\forall s\in\mathbb{R}\mbox{ such that }s<\Lambda-\tfrac12\mbox{ and } s\leq 1,
\end{align}
where $r$ is equal to one for $\Lambda-\tfrac12\in (1,\tfrac32]$ and equal to zero else.
If further $\{\mathcal{T}_h\}$ is $O(h^2)$ irregular, then
\begin{equation}\label{T51s}
\|\bar u-\bar u_h\|_{L^2(\Gamma)}+\|\bar y-\bar y_h\|_{L^2(\Omega)}\le ch^{s}\
\forall s\in\mathbb{R}\mbox{ such that
}s<\min\{\tfrac32,\Lambda-\tfrac12,2
{\lambda}\}.
\end{equation}
\end{theorem}
\begin{remark}\label{sec:rem_ext}
	We only consider the case $a<0<b$. This is because it is known that for those corners such that $\Lambda_j>1$ we have that $\partial_n\bar\varphi(x_j)=0$. In the case $a<0<b$, the projection formula \eqref{eq:proj} implies that in a neighbourhood of $x_j$, the optimal control will satisfy $\bar u(x) = -\partial_n\bar\varphi(x)$, and hence its regularity will be determined by that of the adjoint state. If $0\not\in[a,b]$, then the same projection formula implies that in a neighborhood of $x_j$, $\bar u(x)$ will be equal to some of the control bounds. If we suppose, as in Assumption \ref{assump1} that this property is inherited by the solutions of the discrete approximations, we have that the
	conclusions of Theorem \ref{T5main} remain valid.
\end{remark}
The proof of Theorem \ref{T5main} is postponed to Section \ref{sec:proof1}. Since $\Lambda>1$ and $\lambda>1/2$, we always have a convergence rate greater than $1/2$. This is a real improvement compared the unconstrained case since in the latter it may happen that the convergence rates tend to zero as the largest interior angle tend to $2\pi$. However, one may ask for a justification of Assumption \ref{assump1}. In Lemma \ref{L5.9} we will see that there exist constants $\tilde\rho_{1,j}$ and $\tilde\rho_{1,2}$ greater than zero for all $j\in\mathbb{H}$, and a constant $h_0>0$ such that
\begin{equation}\label{eq:nei}
	\bar u_h(x_{h,i})=\bar u(x_{h,i})\quad \text{for all nodes }x_{h,i} \text{ with }|x_{h,i}-x_j|\in [\tilde{\rho}_{1,j} h |\log h|^{1/2} ,\tilde{\rho}_{2,j}].
\end{equation}
Thus, we could relax Assumption \ref{assump1} to an $h$-dependent neighborhood of those corners $x_j$ with $j\in\mathbb{H}$. Moreover, due to \eqref{eq:nei}, it is even possible to show the following improved result in non-convex domains without any structural assumption on the discrete optimal control, i.e., we can always expect a convergence rate close to $1/2$ in non-convex domains.
\begin{theorem}\label{T5.1}
	Suppose $y_\Omega\in W^{1,p^*}(\Omega)$ for some $p^*>2$,
and assume that the optimal control has a finite number of kink points.
Then it holds	\begin{equation}
	\|\bar u-\bar u_h\|_{L^2(\Gamma)}+\|\bar y-\bar y_h\|_{L^2(\Omega)}\le ch^{1/2}|\log h|^{1/4}.
	\end{equation}
\end{theorem}
The proof of Theorem \ref{T5.1} is given in Section \ref{sec:proof2}.
\subsection{Proof of Theorems \ref{main:cc} and \ref{T5main}}\label{sec:proof1}
The results of Theorem \ref{main:cc}, and Theorem \ref{T5main} for $\lambda>1$ directly follow from the general error estimate given in Theorem \ref{T3.2}, the estimates for the adjoint state provided in Section \ref{sec:4} in Lemmata \ref{L4.3} and \ref{L4.4} and the error estimates for the control and the state established below in Lemma \ref{L5.1}.

\begin{lemma}\label{L5.1}Suppose $y_\Omega\in H^{t}(\Omega)
{\cap L^2(\Omega)}$ for all $t<\min\{1,\lambda-1\}$
and assume that the optimal control has a finite number of kink points.
Then
	\begin{alignat*}{2}  \| \bar u-u_h^*\|_{L^2(\Gamma)}  \leq c h^s&\quad&& \forall s<\min\{\tfrac32,\Lambda-\tfrac12\}, \\
	\|\bar y -S_h\bar u\|_{L^2(\Omega)}\leq c h^s &\quad&& \forall s<
	\min\{\tfrac32,\Lambda-\tfrac12,2
{\lambda}\}.
	\end{alignat*}
\end{lemma}
\begin{proof}
The proof starts exactly following the lines of the proof of Lemma \ref{L4.1}, using the regularity stated in Lemma \ref{L2.4}. In this way, if $s<1$ we again obtain the desired estimate for $u_h^*$, as defined in \eqref{eq:u*1}, from \cite[Eq. (7.10)]{Casas-Raymond2006}. If $s\in[1,\frac32)$, $u_h^*$ is given by \eqref{eq:u*2}. Since control constraints are now present, we have to derive error estimates for the modified Lagrange interpolant. To this end, let us consider two adjacent boundary elements $E_{j-1}$ and $E_j$ belonging to $\mathcal{E}_h$ which are determined by the line segments $(x_\Gamma^{j-1}, x_\Gamma^{j})$ and $(x_\Gamma^{j}, x_\Gamma^{j+1})$, respectively. Since we assume a finite number of kink points of $\bar u$ due to the projection formula \eqref{eq:proj}, we have to deal with the following situations (at least for $h$ small enough): First, no kink is contained in $E_{j-1}\cup E_j$, second, there is exactly one kink of $\bar u$ in $E_{j-1}\cup E_j$ due to the projection formula. In the first case, we have that $u_{h}^*$ coincides with the Lagrange interpolant on $E_{j-1}\cup E_j$ such that the desired estimate on these elements is obtained by standard discretization error estimates for the Lagrange interpolant employing the regularity results from Lemma \ref{L2.4}, i.e.,
\begin{equation}\label{eq:int1}
\| \bar u-u_h^*\|_{L^2(E_{j-1}\cup E_j)}  \leq c h^s|\bar u|_{H^s(E_{j-1}\cup E_j)}
\end{equation}
with $s<\min\{3/2,\Lambda-1/2\}$.
In the second case, we can assume without loss of generality that $u_{h,j-1}^*=b=\bar u(x_\Gamma^{j-1})$, $u_{h,j}^*=b\neq\bar u(x_\Gamma^{j})$ and $u_{h,j+1}^*=\bar u(x_\Gamma^{j+1})\in(a,b)$. Thus, $u_h^*$ is equal to $b=\bar u(x_\Gamma^{j-1})$ on $E_{j-1}$. Using the regularity of the optimal control $\bar u\in H^s(\Gamma)\hookrightarrow C^{0,s-1/2}(\Gamma)$ with $s<\min\{3/2,\Lambda-1/2\}$ from Lemma \ref{L2.4}, we now estimate the interpolation error on each of the elements $E_{j-1}$ and $E_j$. For the error on $E_{j-1}$ we obtain by means of the H\"older continuity of $\bar u$
\begin{equation}
\| \bar u-u_h^*\|_{L^2(E_{j-1})}=\| \bar u-\bar u(x_\Gamma^{j-1})\|_{L^2(E_{j-1})}\leq c|x_\Gamma^{j-1}-x_\Gamma^{j}|^{s-1/2}|E_{j-1}|^{1/2}\sim c h^s.
\end{equation}
Next, recall that the nodal basis function associated with $x_\Gamma^j$ is denoted by $e_j$. Then we deduce for the error on $E_{j}$
\begin{align}
\| \bar u-u_h^*\|_{L^2(E_{j})}&=\| (e_j+e_{j+1})\bar u-\bar u(x_\Gamma^{j-1})e_j-\bar u(x_\Gamma^{j+1})e_{j+1}\|_{L^2(E_{j})}\notag\\
&=\| (\bar u-\bar u(x_\Gamma^{j-1}))e_j+(\bar u- \bar u(x_\Gamma^{j+1}))e_{j+1}\|_{L^2(E_{j})}\notag\\
&\leq\|\bar u-\bar u(x_\Gamma^{j-1})\|_{L^2(E_{j})}+\|\bar u- \bar u(x_\Gamma^{j+1})\|_{L^2(E_{j})}\notag\\
&\leq c\left(|x_\Gamma^{j-1}-x_\Gamma^{j+1}|^{s-1/2}|E_{j}|^{1/2}+|x_\Gamma^{j}-x_\Gamma^{j+1}|^{s-1/2}|E_{j}|^{1/2}\right)\notag\\
&\sim c h^s,\label{eq:int2}
\end{align}
where we again used the H\"older continuity of $\bar u$. Since we assume a finite number of kink points, the desired interpolation error estimate for $u_h^*$ on $\Gamma$ in case that $s\in[1,\frac32)$ is just a combination of \eqref{eq:int1}--\eqref{eq:int2}.

Since the optimal control $\bar u$ belongs at least to $H^{1/2}(\Gamma)$, the optimal state $\bar y$ is a weak solution such that we can rely on standard techniques for the derivation of the second estimate of the assertion. More precisely, by employing the regularity of $\bar y\in H^{t+1}(\Omega)$ with $t<\min\{1,\Lambda-1,\lambda\}$ and $\bar u\in H^r(\Gamma)$ with $r<\min\{\frac32,\Lambda-\frac12\}$ from Lemma \ref{L2.4}, an application of a duality argument, cf. for instance \cite{Bartels2004}, yields
	\[
	\|\bar y -S_h\bar u\|_{L^2(\Omega)}\le ch^{t+\min\{1,r+\tfrac12,\lambda\}}\le ch^{t+\min\{1,\lambda\}}\le c h^s,
	\]
	where $s<\min\{\frac32,\Lambda-\frac12,2\lambda\}$. For the last two steps notice that $\Lambda>1$ and $\lambda>1/2$.
\end{proof}

Since $\Lambda\geq\lambda$, a straightforward application of Theorem \ref{T3.2}, and Lemmata \ref{L5.1}, \ref{L4.3}
and \ref{L4.4} leads to an order of convergence identical to the one we have for unconstrained problems.
Notice that Lemmata \ref{L4.3} and \ref{L4.4} can be used since bounds on the control do not play any role there.
Thus, Theorem \ref{main:cc} and Theorem \ref{T5main} for $\lambda>1$ are proved.

For the results of Theorem \ref{T5main}, in case that $\lambda<1$ and Assumption \ref{assump1} is valid, we use the above error estimates for the control and the state, and we show in Lemmata \ref{L5.3} and \ref{L5.7} below how to improve the result for the adjoint state. Then an adaptation of the general error estimate, see Theorem \ref{T5.2}, which we are going to prove next, can finally be used to combine these results.
Let us define
\[
\tilde \Gamma:=\{x\in \Gamma:\ |x-x_j|<\tilde \rho_j \mbox{ if }j\in\mathbb{H}\}.
\]
Moreover, let
\[
V_h := \{u_h\in U_h:\ u_h\equiv 0\mbox{ on }\tilde\Gamma\}.
\]
Under the structural Assumption \ref{assump1} it is clear that $e_h=u_h^*-\bar u_h\in V_h$, so we have the following modification of the general error estimate \eqref{eq:general_estimate}.
\begin{theorem}\label{T5.2}Suppose Assumption \ref{assump1} holds. Then
	\begin{align*}
	\|\bar u-\bar u_h&\|_{L^2(\Gamma)}+\|\bar y-\bar y_h\|_{L^2(\Omega)}\\
	&\le c\left(\|\bar u- u_h^*\|_{L^2(\Gamma)}+\|\bar y-S_h\bar u\|_{L^2(\Omega)}+\sup_{\psi_h\in V_h}\frac{\left|(\nabla \bar\varphi,\nabla S_h\psi_h)_{L^2(\Omega)}\right|}{\|\psi_h\|_{L^2(\Gamma)}}\right).
	\end{align*}
\end{theorem}
\begin{proof}
Since $e_h=u_h^*-\bar u_h\in V_h$ due to Assumption \ref{assump1}, the result can be obtained in the same way as in the proof of Theorem \ref{T3.2} just by replacing
\[
(\nabla \bar\varphi,\nabla S_he_h)_{L^2(\Omega)}\leq \sup_{\psi_h\in U_h}\frac{\left|(\nabla \bar\varphi,\nabla S_h\psi_h)_{L^2(\Omega)}\right|}{\|\psi_h\|_{L^2(\Gamma)}}\|e_h\|_{L^2(\Gamma)}
\]
in \eqref{eq:generr5} by
\[
(\nabla \bar\varphi,\nabla S_he_h)_{L^2(\Omega)}\leq \sup_{\psi_h\in V_h}\frac{\left|(\nabla \bar\varphi,\nabla S_h\psi_h)_{L^2(\Omega)}\right|}{\|\psi_h\|_{L^2(\Gamma)}}\|e_h\|_{L^2(\Gamma)}.
\]
\end{proof}

Next, we are concerned with discretization error estimates for the critical term in the general estimate of Theorem \ref{T5.2}. First, we deal with estimates for general quasi-uniform meshes. Afterwards we show improved estimates if we assume $O(h^2)$-irregular meshes.

\begin{lemma}\label{L5.3} Let $y_\Omega\in W^{1,p^*}(\Omega)$ for some $p^*>2$. Then there is the estimate
\[
\sup_{\psi_h\in V_h}\frac{\left|(\nabla \bar\varphi,\nabla S_h\psi_h)_{L^2(\Omega)}\right|}{\|\psi_h\|_{L^2(\Gamma)}}\le
{ch^{s}
{|\log h|^r}\quad \forall s\in\mathbb{R}\mbox{ such that }s<\Lambda-\frac12\mbox{ and } s\leq 1,}
\]
where $r$ is equal to one for $\Lambda-1/2\in (1,3/2]$ and equal to zero else.
\end{lemma}

\begin{proof}
To be able to localize the effects in the neighborhood of all corners $x_j$ with $j\in\mathbb{H}$, we introduce a cut-off function $\eta_1$ which is equal to one in a fixed neighborhood of these corners and decays smoothly. In addition, we set $\eta_0=1-\eta_1$. Then we infer for the quantity of interest
\begin{align}
	\sup_{\psi_h\in V_h}&\frac{\left|(\nabla \bar\varphi,\nabla S_h\psi_h)_{L^2(\Omega)}\right|}{\|\psi_h\|_{L^2(\Gamma)}}\notag\\
	&\le
	\sup_{\psi_h\in V_h}\frac{\left|(\nabla (\eta_0\bar\varphi),\nabla S_h\psi_h)_{L^2(\Omega)}\right|}{\|\psi_h\|_{L^2(\Gamma)}}+
	\sup_{\psi_h\in V_h}\frac{\left|(\nabla (\eta_1\bar\varphi),\nabla S_h\psi_h)_{L^2(\Omega)}\right|}{\|\psi_h\|_{L^2(\Gamma)}}\label{eq:sup}
\end{align}
For the first term on the right hand side of this inequality, we directly apply Lemma \ref{L4.3}
to conclude
\begin{align}
\sup_{\psi_h\in V_h}\frac{\left|(\nabla (\eta_0\bar\varphi),\nabla S_h\psi_h)_{L^2(\Omega)}\right|}{\|\psi_h\|_{L^2(\Gamma)}}&\leq ch^{s}
|\log h|^r\\ & \forall s\in\mathbb{R}\mbox{ such that }s<\Lambda-\tfrac12\mbox{ and } s\leq 1,
\end{align}
where $r$ is equal to one for $\Lambda-1/2\in (1,3/2]$ and equal to zero else, having in mind the regularity results of Lemma \ref{C4.1r} for the adjoint state and noting that the singular terms coming from the corners $x_j$ with $j\in\mathbb{H}$ do not have any influence due to the cut-off function $\eta_0$. To deal with the second term in \eqref{eq:sup}, let $\tilde S_h$ denote the extension operator which extends a piecewise linear function $\psi_h$ on the boundary by zero to a function in $Y_h$.
Thus, $\tilde S_h\psi_h$ is equal to zero in $\tilde \Omega:=\{x\in \Omega:\ |x-x_j|<\tilde \rho_j/2 \mbox{ if }j\in\mathbb{H}\}$ for any $\psi_h\in V_h$.
Moreover, let $R_h$ be the operator that maps a function in $H^1_0(\Omega)$ to its Ritz-projection in $Y_{0,h}$. Due the properties of the discrete harmonic extension $S_h$
and the Ritz-projection $R_h$, we obtain
\begin{align}
	(\nabla (\eta_1\bar\varphi),\nabla S_h\psi_h)_{L^2(\Omega)}&=(\nabla (\eta_1\bar\varphi-R_h (\eta_1\bar\varphi)),\nabla S_h\psi_h)_{L^2(\Omega)}\notag\\
	&=(\nabla (\eta_1\bar\varphi-R_h(\eta_1\bar\varphi)),\nabla \tilde S_h\psi_h)_{L^2(\Omega)}\notag\\
	&=(\nabla (\eta_1\bar\varphi-R_h(\eta_1\bar\varphi)),\nabla \tilde S_h\psi_h)_{L^2(\Omega\backslash\tilde{\Omega})}.\label{eq:5.6.1}
\end{align}
By applying the H\"older inequality, local $W^{1,\infty}$-discretization error estimates for the Ritz-projection from \cite[Corollary 1]{demlow2012}, and \eqref{eq:local1}, we obtain
\begin{align}
	(\nabla (\eta_1\bar\varphi&-R_h(\eta_1\bar\varphi)),\nabla \tilde S_h\psi_h)_{L^2(\Omega\backslash\tilde{\Omega})}\\ &\leq\|\nabla (\eta_1\bar\varphi-R_h(\eta_1\bar\varphi))\|_{L^\infty(\Omega\backslash\tilde{\Omega})}\|\nabla \tilde S_h\psi_h\|_{L^1(\Omega\backslash\tilde{\Omega})}\notag\\
	&\leq c\left(\|\nabla(\eta_1\bar\varphi-I_h(\eta_1\bar{\varphi}))\|_{L^\infty(\Omega\backslash\tilde{\Omega}')}+\|\eta_1\bar\varphi-R_h(\eta_1\bar\varphi)\|_{L^2(\Omega)}\right)\|\psi_h\|_{L^1(\Gamma)},\label{eq:5.6.2}
\end{align}
where $\tilde \Omega':=\{x\in \Omega:\ |x-x_j|<\tilde \rho_j/4 \mbox{ if }j\in\mathbb{H}\}$.
Having regard to the regularity results for the optimal adjoint state from Lemma \ref{C4.1r} and by using standard interpolation error estimates and a standard finite element error estimate, we deduce
\begin{equation}\label{eq:5.6.3}
	\|\nabla(\eta_1\bar\varphi-I_h(\eta_1\bar{\varphi}))\|_{L^\infty(\Omega\backslash\tilde{\Omega}')}+\|\eta_1\bar\varphi-R_h(\eta_1\bar\varphi)\|_{L^2(\Omega)}\leq c\left(h+h^s\right)
\end{equation}
which is valid for all $s\in\mathbb{R}$ such that $s<2\lambda$ and $s\leq2$. Combining \eqref{eq:sup}--\eqref{eq:5.6.3} ends the proof.
\end{proof}

\begin{lemma}\label{L5.7} Let $y_\Omega\in W^{1,p^*}(\Omega)$ for some $p^*>2$ and suppose further that $\{\mathcal{T}_h\}$ is a family of $O(h^2)$-irregular meshes. Then it holds
\[
	\sup_{\psi_h\in V_h}\frac{\left|(\nabla \bar\varphi,\nabla S_h\psi_h)_{L^2(\Omega)}\right|}{\|\psi_h\|_{L^2(\Gamma)}}\le ch^{s}\] 
for all $s\in\mathbb{R}\mbox{ such that }s<\min\{\Lambda-\tfrac12,2
{\lambda}\}\text{ and } s\leq \tfrac32.$
\end{lemma}
\begin{proof}
As before, we introduce the circular sectors
	\begin{align*}
	&\tilde \Omega:=\{x\in \Omega:\ |x-x_j|<\tilde \rho_j/2 \mbox{ if }j\in\mathbb{H}\},\\
	&\tilde \Omega':=\{x\in \Omega:\ |x-x_j|<\tilde \rho_j/4 \mbox{ if }j\in\mathbb{H}\}.
	\end{align*}
	For technical reasons we also need the circular sector
	\[\tilde \Omega'':=\{x\in \Omega:\ |x-x_j|<\tilde \rho_j/8 \mbox{ if }j\in\mathbb{H}\}.
	\]
	Let the operators $\tilde S_h$ and $R_h$ be defined as in the proof of Lemma \ref{L5.3}. Moreover, let $\eta_1$ be a smooth cut-off function which is equal to one in $\tilde{\Omega}''$ with $\operatorname{supp} \eta_1\subset \tilde{\Omega}'$.
In addition, we choose $\eta_1$ such that $\operatorname{supp}I_h\eta_1\subset \tilde\Omega'$ which is possible without any restriction for $h$ small enough. We set $\eta_0:=1-\eta_1$. Analogously to the foregoing proof, we infer
	\begin{align}
	(\nabla &\bar\varphi,\nabla S_h\psi_h)_{L^2(\Omega)}=(\nabla (\eta_1\bar\varphi),\nabla S_h\psi_h)_{L^2(\Omega)}
	+(\nabla (\eta_0\bar\varphi),\nabla S_h\psi_h)_{L^2(\Omega)}\notag\\
	&=(\nabla (\eta_1\bar\varphi-R_h(\eta_1\bar \varphi)),\nabla \tilde S_h\psi_h)_{L^2(\Omega\backslash\tilde{\Omega})}+(\nabla (\eta_0\bar\varphi-I_h(\eta_0\bar{\varphi})),\nabla S_h\psi_h)_{L^2(\Omega)}.\label{eq:5.7.1}
	\end{align}
	Observe that $\eta_0\bar{\varphi}$ is equal to zero in a fixed neighborhood of all corners $x_j$ with
$j\in\mathbb{H}$. Consequently, Lemma \ref{lemma:DGH}, applied as in the proof of Lemma \ref{L4.4}, yields for the latter term in \eqref{eq:5.7.1}
	\[
		(\nabla (\eta_0\bar\varphi-I_h(\eta_0\bar{\varphi})),\nabla S_h\psi_h)_{L^2(\Omega)}\leq c h^{s}\|\psi_h\|_{L^2(\Gamma)}
	\]
with $s\in\mathbb{R}$ such that $s<\Lambda-\frac12$ and $s\leq \frac32$.
By applying the H\"older inequality, local $W^{1,\infty}$-discretization error estimates for the Ritz-projection from \cite[Corollary 1]{demlow2012}, and \eqref{eq:local1}, we obtain for the first term in \eqref{eq:5.7.1}
	\begin{align*}
		(\nabla &(\eta_1\bar\varphi-R_h(\eta_1\bar \varphi)),\nabla \tilde S_h\psi_h)_{L^2(\Omega\backslash\tilde{\Omega})}\\ &\leq\|\nabla (\eta_1\bar\varphi-R_h(\eta_1\bar \varphi))\|_{L^\infty(\Omega\backslash\tilde{\Omega})}\|\nabla \tilde S_h\psi_h\|_{L^1(\Omega\backslash\tilde{\Omega})}\\ &\leq c\left(\|\nabla(\eta_1\bar\varphi-I_h(\eta_1\bar\varphi))\|_{L^\infty(\Omega\backslash\tilde{\Omega}')}+\|\eta_1\bar\varphi-R_h(\eta_1\bar\varphi)\|_{L^2(\Omega\backslash\tilde{\Omega}')}\right)\|\psi_h\|_{L^1(\Gamma)}\\
		&\leq c\|\eta_1\bar\varphi-R_h(\eta_1\bar\varphi)\|_{L^2(\Omega\backslash\tilde{\Omega}')}\|\psi_h\|_{L^1(\Gamma)},
	\end{align*}
	where we used that $\eta_1\bar\varphi$
and $I_h\left(\eta_1 \bar\varphi\right)$ are equal to zero in $\Omega\backslash\tilde{\Omega}'$. Usual error estimates for the Ritz-projection and a standard embedding yield
	\[
			(\nabla (\eta_1\bar\varphi-R_h(\eta_1\bar \varphi)),\nabla \tilde S_h\psi_h)_{L^2(\Omega\backslash\tilde{\Omega})}\leq h^{
{s}}\|\psi_h\|_{L^2(\Gamma)},
	\]
which is valid for all $s\in\mathbb{R}$ such that $s<2\lambda$ and $s\leq2$. This ends the proof.
\end{proof}

Finally, an application of Theorem \ref{T5.2}, and Lemmata \ref{L5.1}, \ref{L5.3} and \ref{L5.7}, yield the results of Theorem \ref{T5main} where $\lambda<1$ and Assumption \ref{assump1} is satisfied.

\subsection{Proof of Theorem \ref{T5.1}}\label{sec:proof2}
In this subsection we show the results of Theorem \ref{T5.1}. That is, we show a convergence rate close to $1/2$ for the optimal controls and states in the constrained case if the domain is non-convex and even if the structural Assumption
\ref{assump1} does not hold.
For that purpose, let us recall that $\{x_j\}$ denotes the corners of $\Gamma$, $\{x_{\Gamma}^i\}$ is the set of boundary nodes of the mesh and $\{e_{i}\}$ is the basis of $U_h$ such that $e_{i}(x_{\Gamma}^k)=\delta_{ik}$. 
Thus, every function $u_h\in U_h$ can be written as
\[
	u_h=\sum_{i=1}^{N(h)}  u_{h,i}e_{i}\quad \text{with } u_{h,i}=u_h(x_{\Gamma}^i).
\]
By testing the discrete variational inequality appropriately, we deduce
\[
\bar u_{h,i}=\begin{cases} a & \text{if } \displaystyle\int_{\Gamma}(\nu \bar u_h -\partial_n^h\bar{\varphi}_h) e_{i}>0,\\ \\
b & \text{if } \displaystyle\int_{\Gamma}(\nu \bar u_h -\partial_n^h\bar{\varphi}_h)e_{i}<0. \end{cases}
\]

\begin{lemma}\label{L5.9}
For each interior angle $\omega_j>\pi$, where $c_{j,1}$ from \eqref{E3.5} is unequal to zero, there are two constants $\tilde{\rho}_{1,j}$ and $\tilde{\rho}_{2,j}$ greater than zero such that
\[
\bar u_h(x_{\Gamma}^i)=\begin{cases} a & \text{if } c_{j,1}>0\\
b & \text{if } c_{j,1}<0 \end{cases}
\]
for all nodes $x_{\Gamma}^i$ with $|x_{\Gamma}^i-x_j|\in [\tilde{\rho}_{1,j} h |\log h|^{1/2} ,\tilde{\rho}_{2,j}]$.
\end{lemma}
\begin{proof}
	In the following we focus only on one non-convex corner $x_j$. Without loss of generality let $c_{j,1}$ be greater than zero. {Hence the normal derivative of $\bar\varphi$ is negative, and the lower bound of the control is active, and $\nu\bar u-\partial_n\bar\varphi>0$ in the vicinity of this corner.} We need to show that there are two constants $\tilde{\rho}_{1,j}$ and $\tilde{\rho}_{2,j}$ such that
	\[\int_{\Gamma}(\nu \bar u_h -\partial_n^h\bar{\varphi}_h) e_{h,i}>0\]
	for all nodes $x_{h,i}$ with $|x_{h,i}-x_j|\in [\tilde{\rho}_{1,j} h |\log h|^{1/2} ,\tilde{\rho}_{2,j}]$.
According to \cite[Theorem 3.4]{Nazarov-Plemenevsky-1994},
we know that
	\[
	c_{j,1}=(\bar y-y_d,\zeta_{j,1})_{L^2(\Omega)}
	\]
where the function $\zeta_{j,1}$
is of the form
	\[
		\zeta_{j,1}=\pi^{-1/2}\xi_jr_j^{-\lambda_j}\sin(\lambda_j \theta_j)+z_{j,1},
	\]
	where $\xi_j$ denotes the cut-off function introduced at the beginning of Section \ref{sec:2} and the function $z_{j,1}$ denotes a function which solves
	\[
		-\Delta z_{j,1}=[\Delta,\xi_j]\pi^{-1/2}r_j^{-\lambda_j}\sin(\lambda_j \theta_j)\text{ in }\Omega,\quad z_{j,1}=0 \text{ on }\Gamma,
	\]
and $[a,b]=ab-ba$ denotes the commutator. According to Theorem \ref{main:cc}, we deduce the existence of an constant $h_0>0$ such that for all $h<0$ there holds
	\begin{align}
		\tilde c_{j,1}&:=(\bar y_h-y_d,\zeta_{j,1})_{L^2(\Omega)}=(\bar y-y_d,\zeta_{j,1})_{L^2(\Omega)}+(\bar y_h-\bar y,\zeta_{j,1})_{L^2(\Omega)}\notag\\
		&\geq (\bar y-y_d,\zeta_{j,1})_{L^2(\Omega)}-c\|\bar y_h-\bar y\|_{L^2(\Omega)}>0\label{eq:5.8.1}
	\end{align}
	due to the assumption $c_{j,1}=(\bar y-y_d,\zeta_{j,1})_{L^2(\Omega)}>0$. Using this result, we will show that the singular part of the function $\tilde \varphi$ which solves
	\[
		-\Delta \tilde \varphi = \bar y_h - y_d \text{ in }\Omega,\quad \tilde \varphi =0\text{ on }\Gamma,
	\]
	behaves like the singular part of $\bar \varphi$.
	Indeed, $\tilde \varphi$ admits the splitting
	\begin{equation}\label{eq:5.8.2}
		\tilde \varphi = \tilde \varphi_{\reg} + \tilde \varphi_{\sing}
	\end{equation}
	according to Lemma \ref{L3.1}. The regular part $\tilde \varphi_{\reg}$ belongs to $W^{2,q}(\Omega)$, at least for some $q>2$, since $\bar y_h-y_d$ belongs to $L^q(\Omega)$ due to the convergence result of Theorem \ref{main:cc}. The singular part $\tilde \varphi_{\sing}$ can be written as
	\[\tilde \varphi_{\sing}=
{\tilde c_{j,1}}\xi_j r_j^{\lambda_j}\sin(\lambda_j\theta_j),\]
	where the constant $\tilde c_{j,1}$ is greater than zero according to \eqref{eq:5.8.1}. Assuming that $|x_{\Gamma}^i-x_j|$ is already small enough such that $\xi_j\equiv 1$ on $\operatorname{supp} e_{i}$, we get by basic calculations
	\begin{align}
		\int_{\Gamma}(\nu \bar u_h -\partial_n^h\bar{\varphi}_h) e_{i}&=\int_{\Gamma}\nu \bar u_he_{i} + \int_{\Gamma}(\partial_n\tilde{\varphi} -\partial_n^h\bar{\varphi}_h) e_{i}\notag\\
&\qquad  - \int_{\Gamma}\partial_n\tilde{\varphi}_{\reg}e_{i}
-\int_{\Gamma}\partial_n\tilde{\varphi}_{\sing}e_{i}\notag\\
		&\geq -\int_{\Gamma}ce_{i} + \int_{\Gamma}(\partial_n\tilde{\varphi} -\partial_n^h\bar{\varphi}_h) e_{i}+\int_{\Gamma}c_{j,1}r_j^{\lambda_j-1}e_{i}\label{eq:5.8.3}
	\end{align}
	where we used that $\|\bar u_h\|_{L^\infty(\Gamma)}\leq \max\{|a|,|b|\}$ and that $\partial_n\tilde{\varphi}_{\reg}$ is uniformly bounded in $L^\infty(\Gamma)$ due to the embedding $W^{2,q}(\Omega)\hookrightarrow W^{1,\infty}(\Omega)$ for $q>2$. As before, let us denote by $\tilde S_h$ the operator which extends any function of $U_h$ to one in $Y_h$ by zero. Also observe that $\bar\varphi_h$ is the Ritz-projection $R_h \tilde\varphi$ of $\tilde\varphi$.
	Then integration by parts, the definition of $\partial_n^h\bar{\varphi}_h$ in \eqref{normal} and \eqref{eq:5.8.2} yield
	\begin{align}
		\int_{\Gamma}(&\partial_n\tilde{\varphi} -\partial_n^h\bar{\varphi}_h) e_{i}=\int_{\Omega}\nabla (\tilde\varphi - R_h\tilde\varphi)\cdot\nabla\tilde S_h e_{i}\notag\\
		&=\int_{\Omega}\nabla (\tilde\varphi_{\reg} - R_h\tilde\varphi_{\reg})\cdot\nabla\tilde S_h e_{i}+\int_{\Omega}\nabla (\tilde\varphi_{\sing} - R_h\tilde\varphi_{\sing})\cdot\nabla\tilde S_h e_{i}\notag\\
		&\leq \|\nabla (\tilde\varphi_{\reg} - R_h\tilde\varphi_{\reg})\|_{L^2(\Omega)}\|\nabla\tilde S_h e_{i}\|_{L^2(\Omega)}+\int_{\Omega}\nabla (\tilde\varphi_{\sing} - R_h\tilde\varphi_{\sing})\cdot\nabla\tilde S_h e_{i}\notag\\
		&\leq \int_{\Gamma}ce_{i}+\int_{\Omega}\nabla (\tilde\varphi_{\sing} - R_h\tilde\varphi_{\sing})\cdot\nabla\tilde S_h e_{i},\label{eq:5.8.4}
	\end{align}
	where we employed a standard discretization error estimate for the Ritz-projection, \eqref{eq:local1} and an inverse inequality in the last step.
	Now we proceed as in the proof of Lemma \ref{L4.3} between \eqref{eq:local8} and \eqref{eq:local14}. Let $\sigma$ and the subdomains $\Omega_J$ be defined as in that proof and let the index $J$ be chosen such that
$\operatorname{supp} \tilde S_he_{i}\subset\Omega_J$. Assume that $|x_\Gamma^i-x_j|\geq c_I h$ with a constant
$c_I$ large enough such that local $W^{1,\infty}$-error estimates for the Ritz-projection are applicable. Then those estimates of \cite[Corollary 1]{demlow2012} and \eqref{eq:local1} yield
	\begin{align}
		\int_{\Omega}&\nabla (\tilde\varphi_{\sing} - R_h\tilde\varphi_{\sing})\cdot\nabla\tilde S_h e_{i}\leq \|\nabla (\tilde\varphi_{\sing} - R_h\tilde\varphi_{\sing})\|_{L^\infty(\Omega)}\|\nabla\tilde S_h e_{i}\|_{L^1(\Omega)}\notag\\
		&\leq c\left(\|\nabla(\tilde\varphi_{\sing}-
I_h\tilde{\varphi}_{\sing})\|_{L^\infty(\Omega_J')}+
d_J^{-2}\|\tilde\varphi_{\sing}-R_h\tilde\varphi_{\sing}\|_{L^2(\Omega_J')
}\right)\|e_{i}\|_{L^1(\Gamma)}\notag\\
		&\leq cd_{J}^{\lambda_j-2}\left(d_{J}^{2-\lambda_j}\|\nabla(\tilde\varphi_{\sing}-I_h\tilde{\varphi}_{\sing})\|_{L^\infty(\Omega_J')}\right.\notag\\
&\qquad\qquad\quad\left.+\|\sigma^{-\lambda_j}(\tilde\varphi_{\sing}-R_h\tilde\varphi_{\sing})\|_{L^2(\Omega_J')}\right)\int_{\Gamma}e_{i}.\label{eq:5.8.5}
	\end{align}
	The second derivatives of the singular part behave like $r_j^{\lambda_j-2}$. Thus, by means of standard interpolation error estimates and the results of \cite[Corollary 3.62]{pfefferer2014}, we infer
	\begin{align}
	d_{J}^{2-\lambda_j}&\|\nabla(\tilde\varphi_{\sing}-I_h\tilde{\varphi}_{\sing})\|_{L^\infty(\Omega_J')}+\|\sigma^{-\lambda_j}(\tilde\varphi_{\sing}-R_h\tilde\varphi_{\sing})\|_{L^2(\Omega_J')}\notag\\
	&\leq c\left(h+h|\log h|^{1/2}\right)\leq c h|\log h|^{1/2}.\label{eq:5.8.6}
	\end{align}
	Combining \eqref{eq:5.8.3}--\eqref{eq:5.8.6}, we obtain	
	\begin{align*}
		\int_{\Gamma}(\nu \bar u_h -\partial_n^h\bar{\varphi}_h) e_{h,i}&\geq \int_{\Gamma}(c_{j,1}r_j^{\lambda_j-1}-c-ch|\log h|^{1/2}d_{J}^{\lambda_j-2})e_{i}\\
		&\geq (c_{j,1}d_J^{\lambda_j-1}-c-ch|\log h|^{1/2}d_{J}^{\lambda_j-2})\int_{\Gamma}e_{i}.
	\end{align*}
	Finally, we observe that $|x_{h,i}-x_j|\sim d_J$. Thus, we are able to choose constants $\tilde{\rho}_{1,j}$ and $\tilde{\rho}_{2,j}$ such that
	\[
		\frac{c_{j,1}}{2}d_j^{\lambda_j-1}-c>0
	\]
	if $|x_{\Gamma}^i-x_j|\leq \tilde{\rho}_{2,j}$ and
	\[
		\frac{c_{j,1}}{2}d_j^{\lambda_j-1}-ch|\log h|^{1/2}d_{J}^{\lambda_j-2}>0
	\]
    if $|x_{\Gamma}^i-x_j|\geq \tilde{\rho}_{1,j} h |\log h|^{1/2}$. This proves the assertion.
\end{proof}

\begin{remark}
	By using the Cauchy-Schwarz inequality, estimates for the Ritz-projection from \cite[Theorem 5.1]{bacuta2003}, \eqref{eq:local1} and an inverse inequality, we infer
	\begin{align*}
		\int_{\Omega}\nabla (\tilde\varphi_{\sing} - R_h\tilde\varphi_{\sing})\cdot\nabla\tilde S_h e_{i}&\leq \|\nabla (\tilde\varphi_{\sing} - R_h\tilde\varphi_{\sing})\|_{L^2(\Omega)}\|\nabla\tilde S_h e_{i}\|_{L^2(\Omega)}\leq ch^{\lambda_j}.
	\end{align*}
	However, this is not enough to show that the discrete optimal control admits one of the control bounds in the direct vicinity of the corner $x_j$, since then
	\[
		\int_{\Gamma}c_{j,1}r_j^{\lambda_j-1}e_{i}\geq c h^{\lambda_j}.
	\]
\end{remark}

Now, we redefine the sets $\tilde \Gamma$ and $\tilde \Omega$ by
\[
\tilde \Gamma:=\{x\in \Gamma:\ |x-x_j|<\tilde{\rho}_{2,j} \mbox{ if }
{j\in\mathbb{H}}\}\]
and\[\tilde \Omega:=\{x\in \Omega:\ |x-x_j|<\tilde{\rho}_{2,j} \mbox{ if }
{j\in\mathbb{H}}\},
\]
and we set again
\[
V_h := \{u_h\in U_h:\ u_h\equiv 0\mbox{ on }\tilde\Gamma\}.
\]
Moreover, let $\Gamma^c:=\Gamma\backslash\tilde \Gamma$.
We have the following modification for the general error estimate
\begin{theorem}\label{T5.11}For the solution of the continuous and the
	discrete optimal control problem we have
	\begin{align}
	\|\bar u-\bar u_h\|_{L^2(\Gamma^c)}+&\|\bar y-\bar y_h\|_{L^2(\Omega)}\le  c\Bigg(\| \bar u-\bar u_h\|_{L^2(\tilde\Gamma)}+\|\bar u- u_h^*\|_{L^2(\Gamma^c)}
\notag\\
	&\left. +\|\bar y-S_h\bar u\|_{L^2(\Omega)}+\sup_{\psi_h\in V_h}\frac{\left|(\nabla \bar\varphi,\nabla S_h\psi_h)_{L^2(\Omega)}\right|}{\|\psi_h\|_{L^2(\Gamma)}}\right).\label{eq:general_estimateg}
	\end{align}
\end{theorem}
{Note that the first term on the left hand side of \eqref{eq:general_estimateg} is a norm with respect to $\Gamma^c$.}

\begin{proof}
	We proceed as in the proof of Theorem \ref{T3.2}. In contrast, we will test the optimality conditions with different functions. For that purpose, let us introduce $\tilde u \in U_{ad}$ and $\tilde u_h\in U^h_{ad}$ by
	\[
	\tilde u=\begin{cases}
	\bar u & \text{a.e. in } \tilde \Gamma\\
	\bar u_h &\text{a.e. in } \Gamma^c
	\end{cases}\quad\text{and}\quad \tilde u_h = \begin{cases}
	\bar u_h & \text{a.e. in } \tilde \Gamma\\
	u_h^* &\text{a.e. in } \Gamma^c
	\end{cases}.
	\]
	Note that $\bar u$ and $u_h^*$ are constant, even coincide, on $\tilde \Gamma$ and that $\bar u_h$ is equal to $\bar u$ at least for all $x_{h,i}$ with $|x_{h,i}-x_j|\leq[\tilde{\rho}_{1,j} h |\log h|^{1/2} ,\tilde{\rho}_{2,j}]$ for
$j\in \mathbb{H}$ according to Lemma \ref{L5.9}. Next, we define the intermediate error $e_h:=\tilde u_h-\bar u_h$, which is equal to zero in $\tilde{\Gamma}$. Then, we obtain
	\begin{align}
	\|\bar u-\bar u_h\|_{L^2(\Gamma^c)}+\|\bar y-\bar y_h\|_{L^2(\Omega)}&\leq \|\bar u- u_h^*\|_{L^2(\Gamma^c)}+\| e_h\|_{L^2(\Gamma)}\notag \\&+\|\bar y-S_h \tilde u_h\|_{L^2(\Omega)}+\|S_he_h\|_{L^2(\Omega)}.\label{eq:generr1g}
	\end{align}
	To deal with the third term, we take into account the continuity of $S_h$:
	\begin{align}
	\|\bar y-S_h \tilde u_h\|_{L^2(\Omega)}&\leq
	\|\bar y-S_h \bar u\|_{L^2(\Omega)}+\|S_h (\bar u- \tilde u_h)\|_{L^2(\Omega)}\notag\\
	&\leq \|\bar y-S_h \bar u\|_{L^2(\Omega)}+ c\| \bar u-\tilde u_h\|_{L^2(\Gamma)}\notag\\
	&\leq \|\bar y-S_h \bar u\|_{L^2(\Omega)}+ c(\| \bar u-\bar u_h\|_{L^2(\tilde\Gamma)}+\| \bar u-u_h^*\|_{L^2(\Gamma^c)}).\label{eq:stabilityg}\end{align}
	Accordingly, we only need estimates for the second and fourth terms in \eqref{eq:generr1}. We begin estimating the second one, but as we will see this also yields an estimate for the fourth term. There holds
	\begin{align}
	\nu\| e_h\|_{L^2(\Gamma)}^2&= \nu(u_h^*-\bar u,e_h)_{L^2(\Gamma^c)}+\nu(\bar u-\bar u_h,e_h)_{L^2(\Gamma)}.\label{eq:generr2g}
	\end{align}
	Next, we consider the second term of~\eqref{eq:generr2g} in detail.
    By adding the continuous and discrete variational inequality with $u=\tilde u$ and $u_h=\tilde u_h$, respectively, we deduce
	\[
	(\nu(\bar u_h-\bar u)+\partial_n\bar \varphi-\partial_n^h\bar\varphi_h,e_h)_{L^2(\Gamma)}+(\nu\bar u-\partial_n\bar\varphi,u_h^*-\bar u)_{L^2(\Gamma)}\geq 0.
	\]
	Rearranging terms and using~\eqref{eq:u_h} leads to
	\begin{equation*}
	\nu(\bar u-\bar u_h,e_h)_{L^2(\Gamma)}\leq (\partial_n\bar \varphi-\partial_n^h\bar\varphi_h,e_h)_{L^2(\Gamma)}.
	\end{equation*}
	Integration by parts (cf.~\cite[Theorem 3.1.1]{nevcas2012direct}),~\eqref{normal},~\eqref{E3.3} and~\eqref{discharmonic} yield
	\begin{align}
	\nu(\bar u-\bar u_h,e_h)_{L^2(\Gamma)}&\leq (\Delta \bar \varphi+(\bar y_h-y_\Omega),S_he_h)_{L^2(\Omega)}+(\nabla (\bar\varphi-\bar \varphi_h),\nabla S_he_h)_{L^2(\Omega)}\notag\\
	&= (\bar y_h -\bar y,S_he_h)_{L^2(\Omega)}+(\nabla \bar\varphi,\nabla S_he_h)_{L^2(\Omega)}\notag\\
	&= (S_h \tilde u_h -\bar y,S_he_h)_{L^2(\Omega)}-\|S_he_h\|_{L^2(\Omega)}^2+(\nabla \bar\varphi,\nabla S_he_h)_{L^2(\Omega)}\label{eq:generr3g}.
	\end{align}
	By collecting the estimates~\eqref{eq:generr2g} and~\eqref{eq:generr3g} we obtain
	\begin{align*}
	\nu\| e_h\|_{L^2(\Gamma)}^2&+\|S_he_h\|_{L^2(\Omega)}^2\notag\\
	\leq& \nu(u_h^*-\bar u,e_h)_{L^2(\Gamma^c)}+(S_h \tilde u_h -\bar y,S_he_h)_{L^2(\Omega)}+(\nabla \bar\varphi,\nabla S_he_h)_{L^2(\Omega)}\notag\\
	\leq &\nu\|u_h^*-\bar u\|_{L^2(\Gamma^c)}\|e_h\|_{L^2(\Gamma)}+\|S_h \tilde u_h -\bar y\|_{L^2(\Omega)}\|S_he_h\|_{L^2(\Omega)}\\
	&+\sup_{\psi_h\in V_h}\frac{\left|(\nabla \bar\varphi,\nabla S_h\psi_h)_{L^2(\Omega)}\right|}{\|\psi_h\|_{L^2(\Gamma)}}\|e_h\|_{L^2(\Gamma)}.
	\end{align*}
	From the Young inequality we can deduce
	\begin{align}
	\| e_h\|_{L^2(\Gamma)}&+\|S_he_h\|_{L^2(\Omega)}\notag\\
	&\leq c\left(\|u_h^*-\bar u\|_{L^2(\Gamma^c)}+\|S_h \tilde u_h -\bar y\|_{L^2(\Omega)}+\sup_{\psi_h\in V_h}\frac{\left|(\nabla \bar\varphi,\nabla S_h\psi_h)_{L^2(\Omega)}\right|}{\|\psi_h\|_{L^2(\Gamma)}}\right).\label{eq:generr4g}
	\end{align}
	Finally, the assertion is a consequence from~\eqref{eq:generr1g}, \eqref{eq:stabilityg} and~\eqref{eq:generr4g}.
\end{proof}
Finally, by observing that
\[
	\|\bar u-\bar u_h\|_{L^2(\tilde \Gamma)}\leq c h^{1/2} |\log h|^{1/4}
\]
according to Lemma \ref{L5.9} and the uniform boundedness of $\bar u$ in $L^\infty(\Gamma)$, we deduce the desired result of Theorem \ref{T5.1} by combining Theorem \ref{T5.11} and Lemmata \ref{L5.1} and \ref{L5.3}.

\section{Numerical experiments}\label{sec:num_exp}

The experiments have been performed with Matlab R2015a on an Intel$^\mathrm{(R)}$ Core$^\mathrm{(TM)}$ i7 CPU 870 @2.93 GHz with 16GB RAM on Windows 7 64 bits. All the scripts and functions have been programmed by us.

To build an example with exactly known solution $\bar u$, we just define $\bar\varphi\in H^1_0(\Omega)$ and compute $\bar u = \proj_{[a,b]}\left(\frac{1}{\nu}\partial_n\bar\varphi\right)$, $\bar y\in H^1(\Omega)$ such that $-\Delta\bar y = 0$ in $\Omega$, $\bar y = \bar u$ on $\Gamma$ and $y_\Omega = \bar y+\Delta\bar\varphi$. In general, it is not possible to compute $\bar y$ exactly, so we will use its finite element approximation to compute an approximation of $y_\Omega$.

Since the aim of the experiment is to measure the order of convergence of the $L^2(\Gamma)$ error in the control variable, we have solved the problems in two quasi-uniform families of $J$ nested meshes obtained by diadic refinement from a rough initial mesh. One of them is built such that it does not have the superconvergence property (see Figure \ref{qmesh}), while the other is obtained using regular refinement, which results in a $O(h^2)$-irregular family which has the superconvergence property (see Figure \ref{smesh}).
The finest mesh has between 1 million and 3.15 million nodes, depending on the geometry of the domain. Notice that these fine meshes induce boundary meshes that only have between 4 thousand and 7 thousand nodes only. To solve the optimization problem, we have used a semismooth Newton method; see \cite{Mateos2017} for the details.
\begin{figure}
  \centering
  \includegraphics[width=0.8\textwidth]{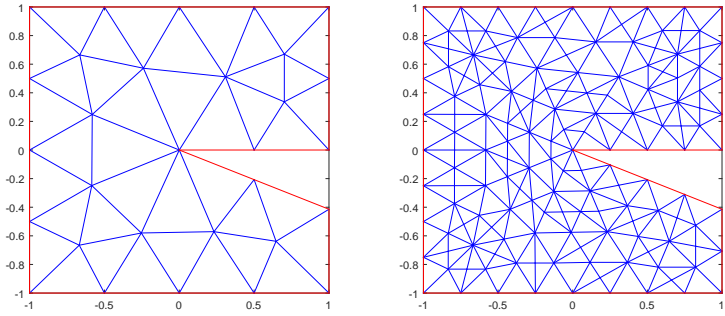}
  \caption{Family of quasi-uniform meshes which is not $O(h^2)$-irregular}\label{qmesh}
\end{figure}
\begin{figure}
  \centering
  \includegraphics[width=0.8\textwidth]{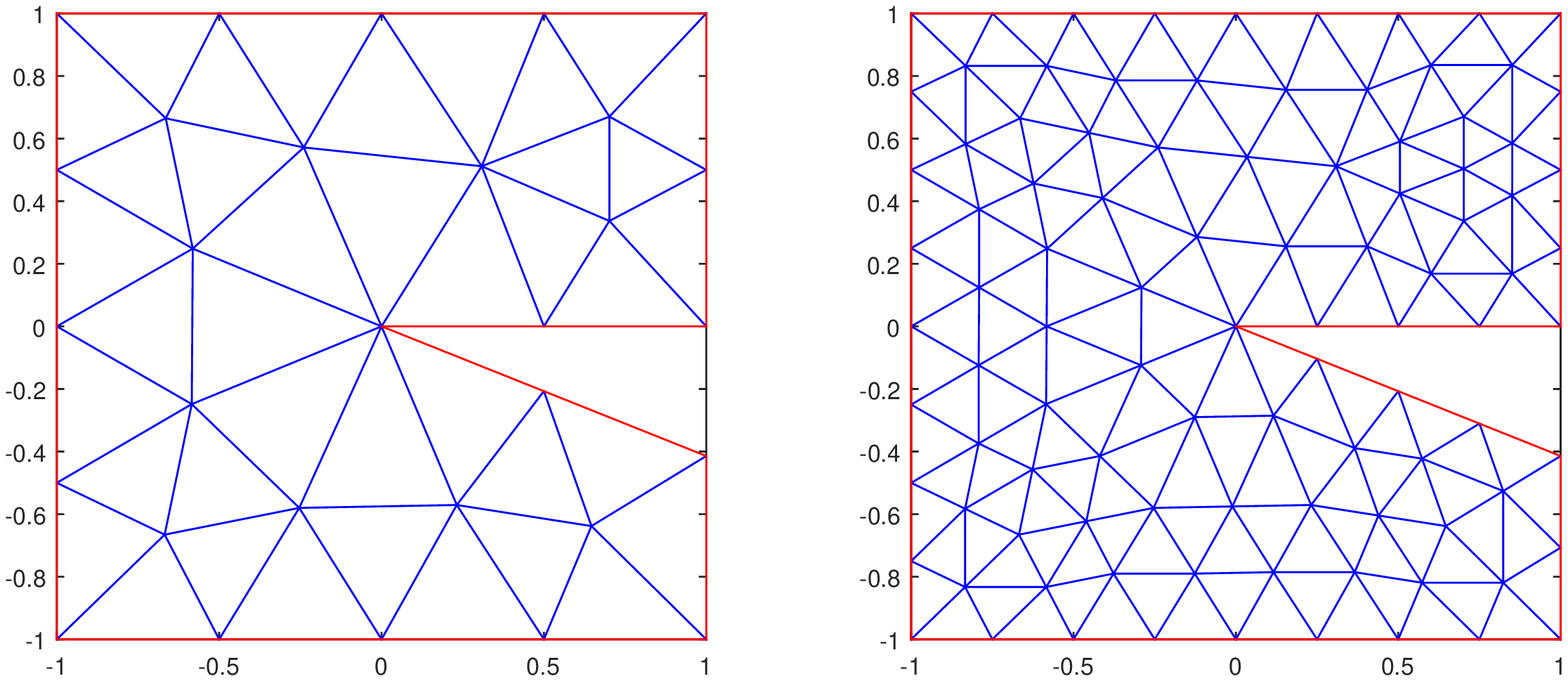}
  \caption{Family of quasi-uniform  $O(h^2)$-irregular meshes}\label{smesh}
\end{figure}

In the examples where the optimal control is continuous, we measure the error at the mesh at level $j=1,\dots,J$ as
\[e_j=\|\bar u_{h_j} -I_{h_j}\bar u\|_{L^2(\Gamma)}\]
where $\bar u_{h_j}$ is the solution of $(P_{h_j})$ and $I_{h_j}:C(\Gamma)\to U_{h_j}$ is the nodal Lagrange interpolation operator. If the exact solution is singular at the point $x_0=(0,0)$, we simply approximate $I_{h_j}\bar u (0,0)\approx \bar u(\epsilon,0)$, for some $\epsilon >0$ small enough.

Since we are using a dyadic refinement strategy, we have that $h_{j+1}=h_j/2$ and we can measure the Experimental Order of Convergence at level $j=2,\dots,J$ as
\[EOC_{j} = \log_2{e_{j-1}}-\log_2{e_{j}}.\]
It is to be expected that $EOC_j$ converges to the Theoretical Order of Convergence ($TOC$) as $j\to\infty$, so for every problem we report on
$EOC := EOC_{J}$ and compare with the corresponding $TOC$.

Let $(\rho,\theta)$ denote the usual polar coordinates in $\mathbb{R}^2$ and define $\Omega$ as the interior of the convex hull of the set of points $\{(0,0),(1,0),(\cos(\omega_1),\sin(\omega_1)\}$ if $\pi/3\leq \omega_1\leq\pi/2$ and $\Omega=\{(x_1,x_2)\in(-1,1)\times(-1,1):\ 0<\theta<\omega_1\}$ for $\pi/2<\omega_1<2\pi$. We will consider the following cases
\begin{enumerate}
  \item $\bar\varphi = r^{\lambda}\sin(\lambda\theta)( \sin(\omega_1)(x-1)+(1-\cos(\omega_1))y )$ if $\omega_1\leq \pi/2$,
  \item $\bar\varphi = r^{\lambda}\sin(\lambda\theta)(1-x_1)(1-x_2)$ if $\pi/2<\omega_1\leq 3\pi/4$,
  \item $\bar\varphi = r^{\lambda}\sin(\lambda\theta)(1-x_1^2)(1-x_2)$ if $3\pi/4<\omega_1\leq 5\pi/4$,
  \item $\bar\varphi = r^{\lambda}\sin(\lambda\theta)(1-x_1^2)(1-x_2^2)$ if $5\pi/4<\omega_1<2\pi$,
\end{enumerate}
where we have tested the value $\lambda = \lambda_1$ for $\pi/3\leq \omega_1<2\pi$, and the special case $\lambda = 2\lambda_1$ for $\pi<\omega_1<2\pi$.
Straightforward calculations show that $\Delta\bar\varphi\in H^t(\Omega)\cap W^{1,p^*}(\Omega)$ for all $t<\lambda-1$ and some $p^*>2$. Also $\partial_\nu\bar\varphi\in H^{s}(\Gamma)$ for all $s<\lambda-1/2$. Hence, for an unconstrained problem $\bar u \in H^{s}(\Gamma)$ for all $s<\lambda-1/2$, which implies that $\bar y\in H^{t}(\Omega)$ for all $t<\lambda_1$ and therefore $y_\Omega=\bar y +\Delta\bar\varphi\in  H^{t}(\Omega)$ for all $t<\lambda_1-1$.  If the problem is constrained, then $\bar u\in L^\infty(\Gamma)$ and therefore $\bar y\in W^{1,p^*}(\Omega)$ for all $2<p^*<p_D$ and $y_\Omega\in W^{1,p^*}(\Omega)$ for some $p^*>2$.

Notice that for the case $5\pi/4<\omega_1<3\pi/2$ we have that $\omega_6=2\pi-\omega_1\in (\pi/2,\pi)$ and for the case $7\pi/4<\omega_1<2\pi$ we have $\omega_7 = 5\pi/2-\omega_1\in (\pi/2,\pi)$, so when we choose $\lambda=\lambda_1$ in the definition of $\bar\varphi$ and solve a constrained problem, the leading singular exponent to be taken into account should be, respectively, $\lambda_6$ or $\lambda_7$. Nevertheless, the exact adjoint state has been chosen in such a way that in the first case $c_{6,m}=0$ and in the second case $c_{7,m}=0$ for $m=1,2,3$, so for this example we need not take this into account.

We fix $\nu = 1$. For constrained problems, we will consider $a=-1/\lambda_1$ and $b=1$. We choose $a$ so the asymptotic behavior of the error shows up for the mesh sizes used. If $|a|$ were two big, the problem would behave like an unconstrained one for our meshes; on the other hand, were $|a|$ too small, we would be approximating an optimal control very similar to a constant and the experimental orders of convergence would be too high for our meshes.

Graphs with the experimental results can be found in figures \ref{F:unc_exp} and \ref{F:con_exp}.
It is remarkable that
experimental results are quite in agreement with theoretical estimates.
\begin{figure}[h!]
\includegraphics[width = 0.4\textwidth]{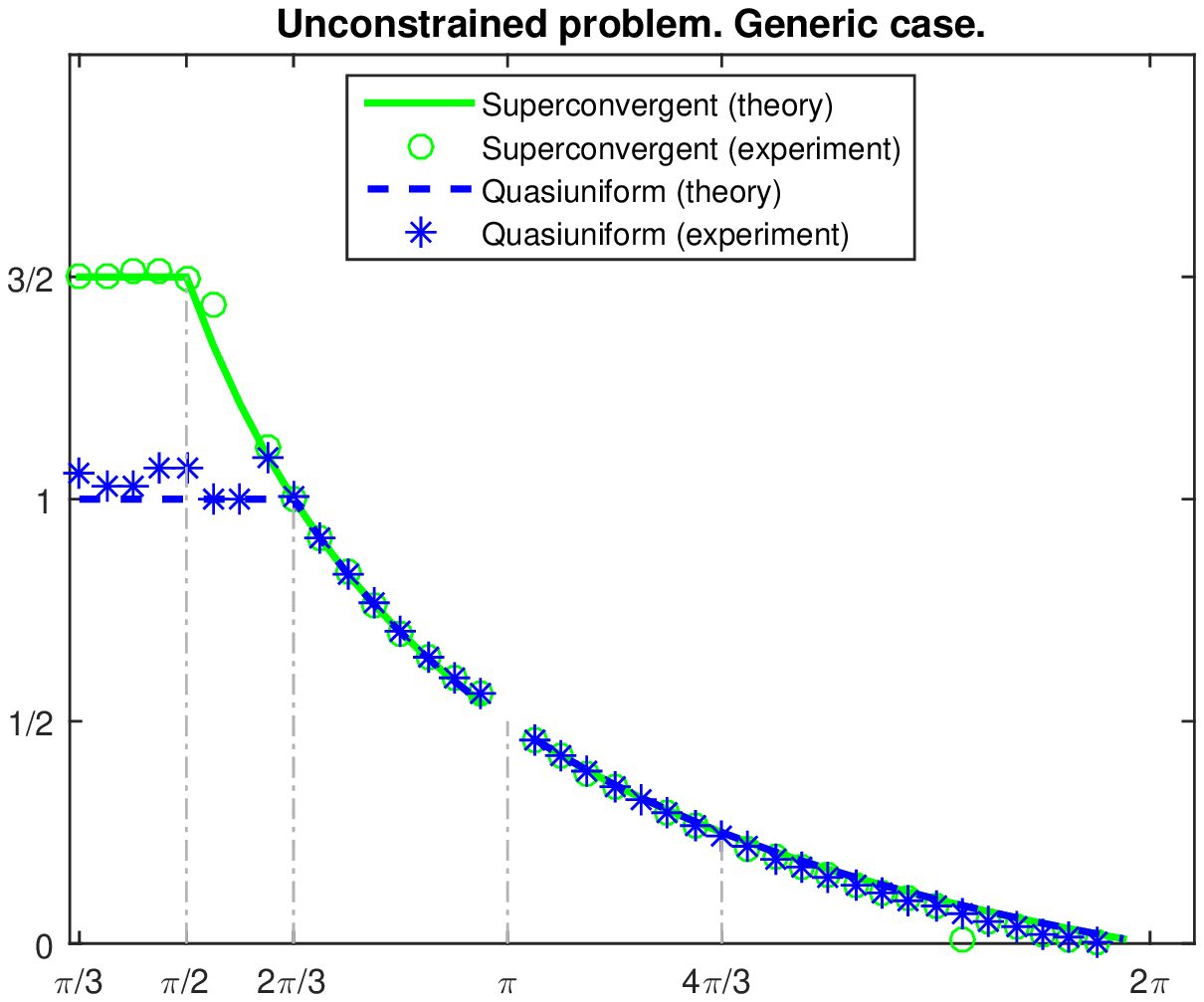}\hspace{0.1\textwidth}
\includegraphics[width = 0.4\textwidth]{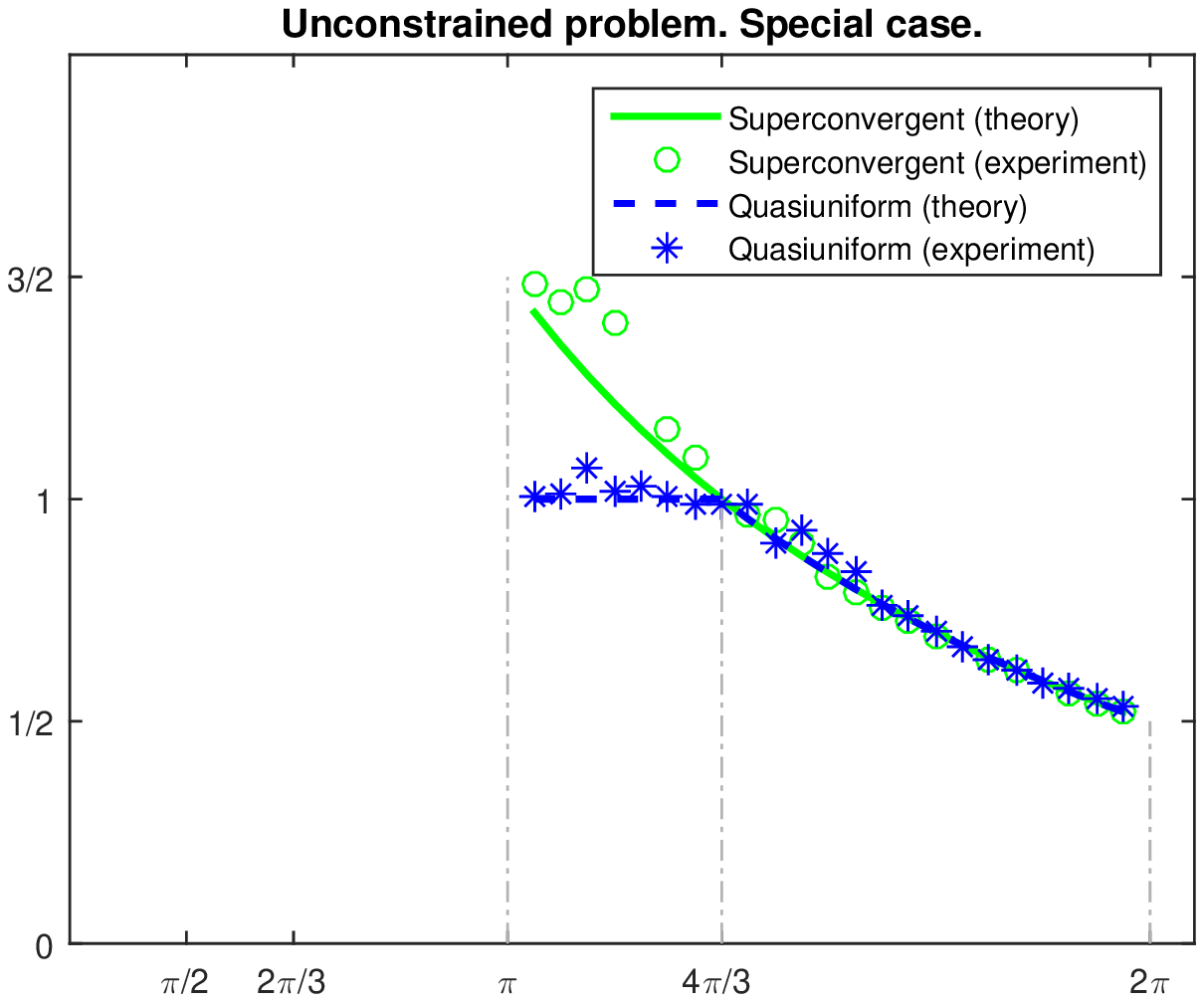}
\caption{Unconstrained problems. Experimental orders of convergence vs biggest angle. Left: generic case. Right: special case.}\label{F:unc_exp}
\end{figure}

\begin{figure}[h!]
\includegraphics[width = 0.4\textwidth]{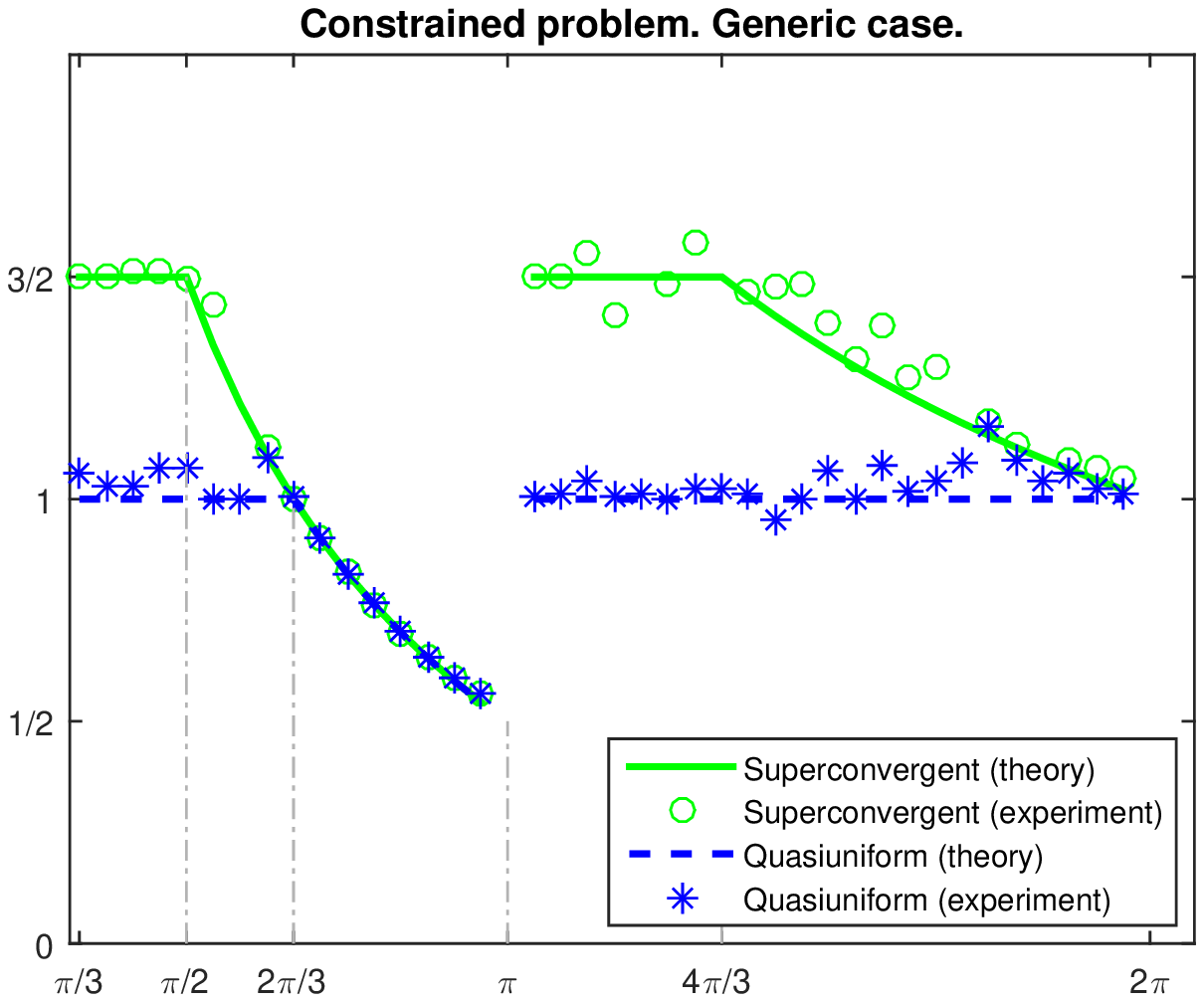}\hspace{0.1\textwidth}
\includegraphics[width = 0.4\textwidth]{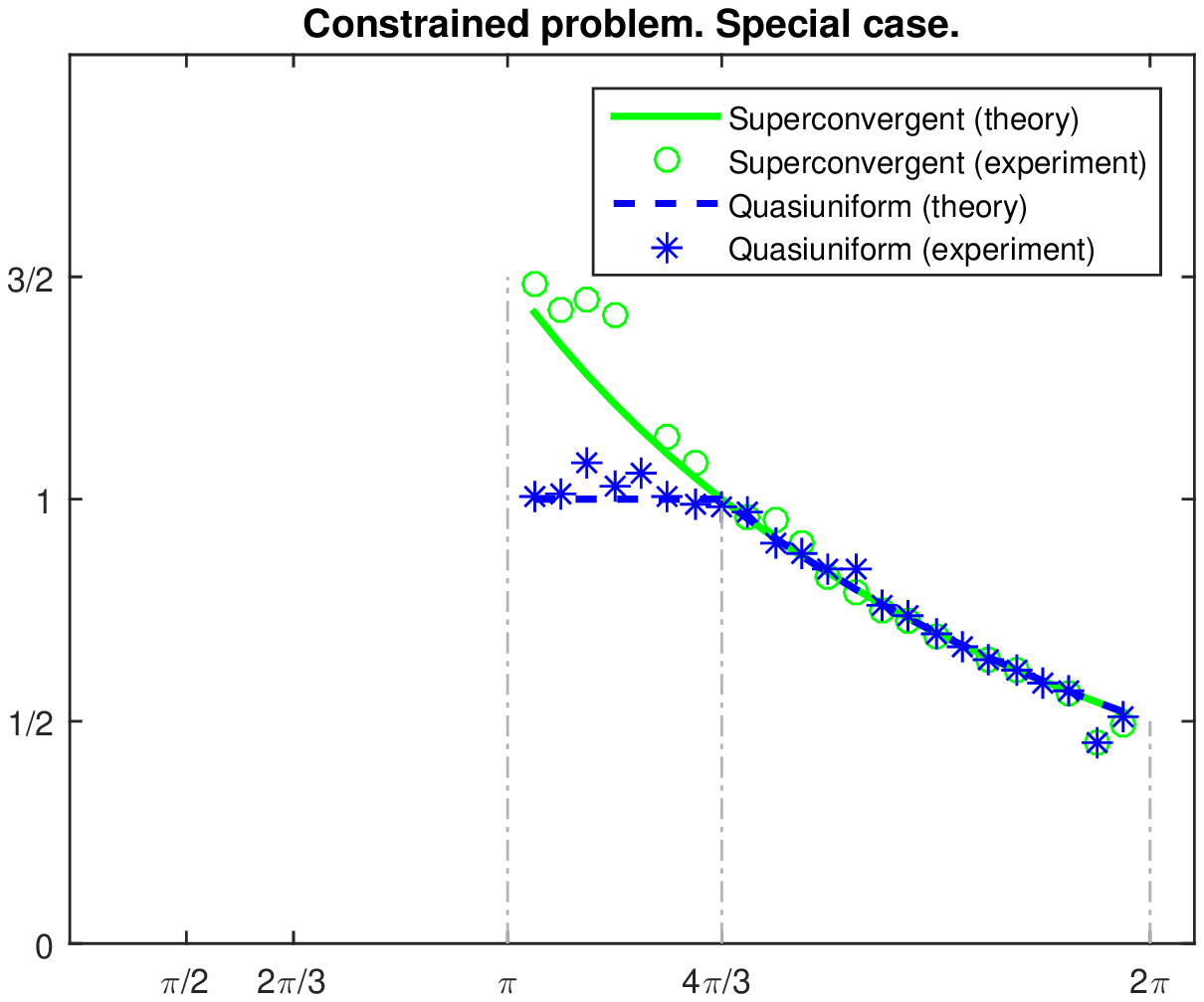}
\caption{Constrained problems. Experimental orders of convergence vs biggest angle. Left: generic case. Right: special case.}\label{F:con_exp}
\end{figure}

\bibliographystyle{AIMS}

\bibliography{references_for_constrained_Dirichlet}

\providecommand{\href}[2]{#2}
\providecommand{\arxiv}[1]{\href{http://arxiv.org/abs/#1}{arXiv:#1}}
\providecommand{\url}[1]{\texttt{#1}}
\providecommand{\urlprefix}{URL }
\begin{thebibliography}{10}

\bibitem{AMPR-2015}
\newblock T.~Apel, M.~Mateos, J.~Pfefferer and A.~R{\"o}sch,
\newblock On the regularity of the solutions of {D}irichlet optimal control
  problems in polygonal domains,
\newblock \emph{SIAM J. Control Optim.}, \textbf{53} (2015), 3620--3641,
\newblock \urlprefix\url{http://dx.doi.org/10.1137/140994186}.

\bibitem{AMPR-2017}
\newblock T.~Apel, M.~Mateos, J.~Pfefferer and A.~R{\"o}sch,
\newblock Mesh grading for the numerical solution of {D}irichlet control
  problems in polygonal domains, 2017,
\newblock In preparation.

\bibitem{Apel-Nicaise-Pfefferer2014}
\newblock T.~Apel, S.~Nicaise and J.~Pfefferer,
\newblock Discretization of the {P}oisson equation with non-smooth data and
  emphasis on non-convex domains,
\newblock \emph{Numer. Methods Partial Differential Equations}, \textbf{32}
  (2016), 1433--1454.

\bibitem{apel2015finite}
\newblock T.~Apel, J.~Pfefferer and A.~R{\"o}sch,
\newblock Finite element error estimates on the boundary with application to
  optimal control,
\newblock \emph{Mathematics of Computation}, \textbf{84} (2015), 33--70.

\bibitem{bacuta2003}
\newblock C.~Bacuta, J.~Bramble and J.~Xu,
\newblock Regularity estimates for elliptic boundary value problems in {B}esov
  spaces,
\newblock \emph{Mathematics of Computation}, \textbf{72} (2003), 1577--1595.

\bibitem{bank2003}
\newblock R.~Bank and J.~Xu,
\newblock Asymptotically exact a posteriori error estimators, part {I}: Grids
  with superconvergence,
\newblock \emph{SIAM Journal on Numerical Analysis}, \textbf{41} (2003),
  2294--2312.

\bibitem{Bartels2004}
\newblock S.~Bartels, C.~Carstensen and G.~Dolzmann,
\newblock {I}nhomogeneous {D}irichlet conditions in a priori and a posteriori
  finite element error analysis,
\newblock \emph{Numerische Mathematik}, \textbf{99} (2004), 1--24,
\newblock \urlprefix\url{http://dx.doi.org/10.1007/s00211-004-0548-3}.

\bibitem{Berggren2004}
\newblock M.~Berggren,
\newblock Approximations of very weak solutions to boundary-value problems,
\newblock \emph{SIAM J. Numer. Anal.}, \textbf{42} (2004), 860--877
  (electronic),
\newblock \urlprefix\url{http://dx.doi.org/10.1137/S0036142903382048}.

\bibitem{Casas-Raymond2006}
\newblock E.~Casas and J.-P. Raymond,
\newblock Error estimates for the numerical approximation of {D}irichlet
  boundary control for semilinear elliptic equations,
\newblock \emph{SIAM J. Control Optim.}, \textbf{45} (2006), 1586--1611
  (electronic),
\newblock \urlprefix\url{http://dx.doi.org/10.1137/050626600}.

\bibitem{casas2003error}
\newblock E.~Casas and F.~Tr{\"o}ltzsch,
\newblock Error estimates for linear-quadratic elliptic control problems,
\newblock in \emph{Analysis and optimization of differential systems},
\newblock Springer, 2003,
\newblock 89--100.

\bibitem{Ciarlet91}
\newblock P.~Ciarlet,
\newblock Basic error estimates for elliptic problems,
\newblock in \emph{Handbook of Numerical Analysis} (eds. P.~Ciarlet and
  J.~Lions), vol. II. Finite Element Methods (Part 1),
\newblock North-Holland, 1991,
\newblock 17--352.

\bibitem{Costabel1988}
\newblock M.~Costabel,
\newblock Boundary integral operators on {L}ipschitz domains: elementary
  results,
\newblock \emph{SIAM J. Math. Anal.}, \textbf{19} (1988), 613--626.

\bibitem{Deckelnick-Gunther-Hinze2009}
\newblock K.~Deckelnick, A.~G{\"u}nther and M.~Hinze,
\newblock Finite element approximation of {D}irichlet boundary control for
  elliptic {PDE}s on two- and three-dimensional curved domains,
\newblock \emph{SIAM J. Control Optim.}, \textbf{48} (2009), 2798--2819,
\newblock \urlprefix\url{http://dx.doi.org/10.1137/080735369}.

\bibitem{demlow2012}
\newblock A.~Demlow, D.~Leykekhman, A.~Schatz and L.~Wahlbin,
\newblock Best approximation property in the ${W}^1_\infty$ norm for finite
  element methods on graded meshes,
\newblock \emph{Mathematics of Computation}, \textbf{81} (2012), 743--764.

\bibitem{horger2013optimal}
\newblock T.~Horger, J.~Melenk and B.~Wohlmuth,
\newblock On optimal ${L}^2$- and surface flux convergence in {FEM},
\newblock \emph{Computing and Visualization in Science}, \textbf{16} (2013),
  231--246.

\bibitem{Mateos2017}
\newblock M.~Mateos,
\newblock Optimization methods for {D}irichlet control problems,
\newblock \emph{Submitted},
\newblock \url{https://arxiv.org/abs/1701.07619}.

\bibitem{Mateos-Neitzel2015}
\newblock M.~Mateos and I.~Neitzel,
\newblock Dirichlet control of elliptic state constrained problems,
\newblock \emph{Comput. Optim. Appl.}, \textbf{63} (2016), 825--853,
\newblock \urlprefix\url{http://dx.doi.org/10.1007/s10589-015-9784-y}.

\bibitem{MayRannacherVexler2013}
\newblock S.~May, R.~Rannacher and B.~Vexler,
\newblock Error analysis for a finite element approximation of elliptic
  {D}irichlet boundary control problems,
\newblock \emph{SIAM Journal on Control and Optimization}, 2585--2611.

\bibitem{melenk2012quasi}
\newblock J.~Melenk and B.~Wohlmuth,
\newblock Quasi-optimal approximation of surface based lagrange multipliers in
  finite element methods,
\newblock \emph{SIAM Journal on Numerical Analysis}, \textbf{50} (2012),
  2064--2087.

\bibitem{Nazarov-Plemenevsky-1994}
\newblock S.~Nazarov and B.~A. Plamenevsky,
\newblock \emph{Elliptic problems in domains with piecewise smooth boundaries},
  vol.~13 of De Gruyter Expositions in Mathematics,
\newblock Walter de Gruyter \& Co., Berlin, 1994,
\newblock \urlprefix\url{http://dx.doi.org/10.1515/9783110848915.525}.

\bibitem{nevcas2012direct}
\newblock J.~Ne{\v{c}}as,
\newblock \emph{Direct methods in the theory of elliptic equations},
\newblock Corrected 2nd edition,
\newblock Monographs and Studies in Mathematics, Springer Berlin Heidelberg,
  2012.

\bibitem{pfefferer2014}
\newblock J.~Pfefferer,
\newblock \emph{Numerical analysis for elliptic Neumann boundary control
  problems on polygonal domains},
\newblock PhD thesis, Universit{\"a}t der Bundeswehr M{\"u}nchen, 2014,
\newblock \url{http://athene.bibl.unibw-muenchen.de:8081/node?id=92055}.

\bibitem{rannacher1982}
\newblock R.~Rannacher and R.~Scott,
\newblock Some optimal error estimates for piecewise linear finite element
  approximations,
\newblock \emph{Mathematics of Computation}, \textbf{38} (1982), 437--445.

\end{thebibliography}
\end{document}